\documentclass[twoside,centertags]{amsart}

\usepackage[cmtip,color,all]{xy}

\usepackage{amsmath}
\usepackage{amsfonts}
\usepackage{amssymb}
\usepackage{amsthm}
\usepackage{graphicx}
\usepackage{color}
\usepackage{mathrsfs}

\newtheorem{thm}[equation]{Theorem}
\newtheorem{cor}[equation]{Corollary}
\newtheorem{lem}[equation]{Lemma}

\newtheorem{prop}[equation]{Proposition}
\theoremstyle{definition}
\newtheorem{defn}[equation]{Definition}
\theoremstyle{remark}
\newtheorem{rem}[equation]{Remark}
\newtheorem{exm}[equation]{Example}

\newcommand{\der}{\mathbb{D}}

\newcommand{\Ar}{\operatorname{Ar}}

\newcommand{\s}{\textsf{s}}

\newcommand{\wk}{K^{W,\operatorname{ob}}}
\renewcommand{\S}{\mathcal{S}}

\newcommand{\DK}{K}
\newcommand{\WK}{K^W}

\newcommand{\Ob}{\operatorname{Ob}}

\newcommand{\C}{\mathcal{C}}

\def\op{\operatorname{op}}

\def\r{\rightarrow} 
\def\rr{\Rightarrow} 

\def\ob{\operatorname{Ob}}

\def\ho{\operatorname{Ho}}
\newcommand{\iso}{\operatorname{iso}}
\newcommand{\eq}{\operatorname{eq}}

\def\diag{\operatorname{diag}}

\def\cat{\mathcal{C}at}  
\def\Dia{\mathcal{D}ia}  
\def\dirf{\mathcal{D}ir_{f}}  
\newcommand{\id}[1]{\operatorname{id}_{#1}} 
\renewcommand{\partial}{d} 

\def\PreDer{\mathtt{PDer}}
\def\eqPreDer{\mathtt{PDer}_{\operatorname{eq}}}
\def\PREDER{\mathtt{\underline{PDer}}^{\operatorname{str}}}
\def\eqPREDER{\mathtt{\underline{PDer}}^{\operatorname{str}}_{\operatorname{eq}}}
\def\Der{\mathtt{Der}}
\def\eqDer{\mathtt{Der}_{\operatorname{eq}}}
\def\DER{\underline{\mathtt{Der}}^{\operatorname{str}}}
\def\eqDER{\underline{\mathtt{Der}}^{\operatorname{str}}_{\operatorname{eq}}}

\def\TOP{\underline{\mathtt{Top}}}
\def\top{\mathtt{Top}}
\def\SSET{\underline{\mathtt{SSet}}}
\def\SSet{\mathcal{SS}et}
\def\sPreDer{\mathtt{PDer}^{\operatorname{str}}}
\def\eqsPreDer{\mathtt{PDer}^{\operatorname{str}}_{\operatorname{eq}}}
\def\sDer{\mathtt{Der}^{\operatorname{str}}}
\def\eqsDer{\mathtt{Der}^{\operatorname{str}}_{\operatorname{eq}}}
\def\eqs1Der{\mathtt{Der}^{\operatorname{str}}_{\operatorname{eq},0}}

\def\muob{{\mu}^{\operatorname{ob}}}

\def\facto{\mathtt{App}}
\def\MOD{\mathscr{MOD}}
\def\HOM{\operatorname{HOM}}
\def\BigDer{\mathscr{DER}}
\def\CAT{\mathscr{CAT}}

\def\To{\longrightarrow}

\newcommand{\dia}{\operatorname{dia}}

\numberwithin{equation}{section}

\newdir{ >}{{}*!/-7pt/@{>}}
\newdir{> }{{}*!/10pt/@{>}}
\newdir{>> }{{}*!/10pt/@{>>}}

\begin{document}

\title{$K$-theory of derivators revisited} %

\author{Fernando Muro}%
\address{Universidad de Sevilla,
Facultad de Matem\'aticas,
Departamento de \'Algebra,
Avda. Reina Mercedes s/n,
41012 Sevilla, Spain}
\email{fmuro@us.es}
\urladdr{http://personal.us.es/fmuro}

\author{George Raptis}%
\address{Fakult\"{a}t f\"{u}r Mathematik, 
Universit\"{a}t Regensburg, 
93040 Regensburg, Germany}
\email{georgios.raptis@mathematik.uni-regensburg.de}

\subjclass[2010]{19D99,55U35}
\keywords{$K$-theory, derivator}

\thanks{The first author was partially supported
by the Andalusian Ministry of Economy, Innovation and Science under the grant FQM-5713, by the Spanish Ministry of Education and
Science under the MEC-FEDER grant  MTM2010-15831, and by the Government of Catalonia under the grant SGR-119-2009. }

\begin{abstract}
We define a $K$-theory for pointed right derivators and show that it agrees
with Waldhausen $K$-theory in the case where the derivator arises from a
good Waldhausen category. This $K$-theory is not invariant under general equivalences
of derivators, but only under a stronger notion of equivalence that is defined by 
considering a simplicial enrichment of the category of derivators. We show that 
derivator $K$-theory, as originally defined, is the best approximation to Waldhausen 
$K$-theory by a functor that is invariant under equivalences of derivators.
\end{abstract}

\maketitle
\setcounter{tocdepth}{1}
\tableofcontents

\section{Introduction}

Recent developments in the theory of derivators have shown that the theory is both sufficiently rich to contend for an independent approach to homotopical algebra 
and its language is very useful in formulating precisely universal properties in homotopy theory. Since models for homotopical algebra typically give rise to derivators, 
the theory reflects a minimalist approach employing basically only purely ($2$-)categorical arguments, albeit technically quite complex at times, to address problems of 
abstract homotopy theory. 

Derivators codify structure lying somewhere between the model and its associated homotopy category, but fairly closer to the model than the homotopy category, 
and restructure the presentation of the homotopy theory defined by the model in a surprisingly efficient way. This intermediate structure involves the collection of all homotopy categories of diagrams 
of various shapes in the model together with the network of restriction functors between them and their adjoint homotopy Kan extensions. The theory of derivators is based 
on an abstract axiomatization of collections of such (homotopy) Kan extensions (and their adjoints) which does not involve any underlying model. On the one hand, it is 
often the case that questions about the model are really questions about the associated derivator and thus they can instructively be handled more abstractly at this level of generality.
On the other hand, for the theory to be successful, one is normally required to supply a large amount of data in order to compensate for the lack of an underlying homotopy 
theory and, consequently, working with these objects can be cumbersome. 

The main problem is to understand how close this passage from the model to its associated derivator 
actually is to being faithful. This paper is a contribution to this problem in connection with Waldhausen $K$-theory regarded as an invariant of homotopy theories.
The proven failure to reconstruct $K$-theory from the triangulated structure of the homotopy category in a way that it satisfies 
certain desirable properties \cite{ankttc} suggested to turn to the more highly structured world of derivators
for such a reconstruction. Indeed, it turned out that the structure of a derivator is rich enough to allow for a natural definition 
of $K$-theory. This was introduced by Maltsiniotis \cite{ktdt} who also conjectured 
that it satisfies an \emph{additivity} property, \emph{agreement} with Quillen $K$-theory and a \emph{localization} 
property. In previous work \cite{ankttd}, we
showed that \emph{agreement} fails for Waldhausen $K$-theory and moreover, that derivator $K$-theory cannot satisfy
both \emph{agreement}
with Quillen $K$-theory and the \emph{localization} property. On the other hand, Cisinski and Neeman \cite{adkt} showed
that \emph{additivity} for derivator $K$-theory holds for triangulated derivators.

The purpose of this paper is twofold. First we define a new $K$-theory of derivators, which we call Waldhausen
$K$-theory, and show that it agrees with Waldhausen $K$-theory for all well-behaved Waldhausen categories (Theorem \ref{agreement}). The proof
rests crucially on the homotopically flexible versions of the $\S_{\bullet}$-construction due to Blumberg-Mandell \cite{aKtaht}
and Cisinski \cite{ikted}. 

The price to be paid for such a strong version of agreement is that this new definition is \emph{provably} not invariant
under equivalences of derivators. In this connection, it is also worthwhile to recall that To\"en-Vezzosi \cite{ktsc} showed that Waldhausen $K$-theory cannot factor 
through the $2$-category of derivators. However, we consider here a simplicial enrichment of the category of derivators which enhances the $2$-categorical structure. 
This leads to a stronger and more refined notion of equivalence of derivators which basically encodes higher coherence and is closer to an equivalence of homotopy 
theories. We show that Waldhausen $K$-theory of derivators is invariant under this stronger notion of equivalence. We think that 
the simplicial enrichment of the category of derivators and the accompanying stronger notion of equivalence have independent interest and may prove useful also in other 
applications of the theory.

There is a natural comparison transformation from Waldhausen $K$-theory to derivator $K$-theory. The second main
result of the paper (Theorem \ref{initial}) says that this comparison transformation is homotopically initial among all natural transformations 
from Waldhausen $K$-theory to a functor which is invariant under equivalences of derivators. 

\

The paper is organized as follows. In Section 2, we review some background material from the ($2$-categorical) theory of derivators and fix some notational conventions. 
In Section 3, we discuss the simplicial enrichment of the category of derivators, the associated notion of strong equivalence, and the comparison with the $2$-categorical
viewpoint. 

Section 4 is concerned with the definition of Waldhausen $K$-theory for derivators and some of its general properties. We present two canonically homotopy equivalent models for 
Waldhausen $K$-theory both of which we are going to use in the paper. Then we show that the Waldhausen $K$-theory of derivators  is invariant under strong
equivalences of derivators and agrees with the Waldhausen $K$-theory of derivable strongly saturated Waldhausen categories. In Section 5, we recall the definition of 
derivator $K$-theory and discuss its dependence on the $2$-categorical theory of derivators. Then we recall the definition of the comparison map from Waldhausen $K$-theory to 
derivator $K$-theory and show that derivator $K$-theory is the best approximation to Waldhausen $K$-theory by a functor that is invariant under equivalences 
of derivators. 

There are several remaining open questions, regarding either the notion of strong equivalence or the $K$-theory of derivators, some of which are briefly mentioned in Section 6. 
The paper ends with two appendices on topics of related interest but which are, strictly speaking, independent of the rest of the paper. In Appendix A, we recall the results 
from the comparison between combinatorial model categories and the $2$-category of derivators due to Renaudin \cite{pcthqd} and discuss some slight improvements with 
an eye towards understanding the comparison with the simplicial category of derivators. Appendix B is concerned with the approximation theorem in $K$-theory, which in 
a version due to Cisinski \cite{ikted} shows that $K$-theory is invariant under derived equivalences, and a partial converse which shows that derived equivalences are 
detected by the homotopy type of the $\S_\bullet$-construction.

\section{Preliminaries on (pre)derivators}\numberwithin{equation}{subsection}

\subsection{Prederivators}
Let $\cat$ denote the $2$-category of small categories. We fix a $1$- and $2$-full sub-$2$-category of diagrams 
$\Dia\subset\cat$ which is closed under all required constructions appearing below (e.g.~~taking opposite categories 
or finite (co)products, passing to comma categories, etc.), see \cite{ktdt} for the precise list of axioms. We think 
of the collection of categories in $\Dia$ as possible shapes for indexing diagrams in other categories. 

The smallest option for $\Dia$ is $\mathcal{P}os_f$, spanned by the finite posets, and the largest option is, of course, $\cat$ itself. 
An intermediate option, which appears prominently in connection with $K$-theory, is the $2$-category of diagrams $\dirf$ spanned
by the finite direct categories. Recall that a \emph{finite direct category} is a small category whose nerve has finitely many non-degenerate 
simplices. This is equivalent to saying that the underlying graph spanned by the non-identity arrows of the category has no 
cycles. Every finite poset is a finite direct category.

A \emph{prederivator} (with domain $\Dia$) is\footnote{The reader should be warned about the slight variations of this definition
that appear in the literature. These pertain to the choice of domain and the different ways of forming the opposite of a $2$-category.}
a strict $2$-functor $\der\colon \Dia^{\op} \to\cat.$ More explicitly, there is a small category $\der(X)$ for every category $X$ in $\Dia$, 
for every functor  $f \colon X\r Y$ in $\Dia$ there is an \emph{inverse image} functor
$$f^{*}= \der(f) \colon\der(Y)\To \der(X),$$
for every natural transformation $\alpha\colon f\rr g$ in $\Dia$,
$$\xymatrix@C=40pt{X\ar@/^15pt/[r]^{f}_{}="a" \ar@/_15pt/[r]_{g}^{}="b"&Y,\ar@{=>}"a";"b"^{\alpha}}$$
there is a natural transformation $\alpha^{*}=\der(\alpha)\colon f^{*}\rr g^{*},$
$$\xymatrix@C=40pt{\der(X)\ar@{<-}@/^15pt/[r]^{f^{*}}_{}="a" \ar@{<-}@/_15pt/[r]_{g^{*}}^{}="b"&\der(Y),\ar@{<=}"b";"a"^{\alpha^{*}}}$$
and all these are required to satisfy the obvious strict $2$-functoriality properties.

A (\emph{$1$-})\emph{morphism of prederivators} $\phi\colon \der\r\der'$ is a pseudo-natural transformation of contravariant $2$-functors, i.e.~
for every $X$ in $\Dia$ there is a functor $$\phi(X)\colon\der(X)\To\der'(X),$$ and for every $f\colon X\r Y$ in $\Dia$
there is a natural isomorphism $\phi(f)$,
\begin{equation}\label{natiso}
\xymatrix{\der(Y)\ar[r]^{\phi(Y)}_-{\;}="a"\ar[d]_{f^{*}}^-{\;}="b"&\der'(Y)\ar[d]^{f^{*}}\\
\der(X)\ar[r]_{\phi(X)}&\der'(X)\ar@{=>}"a";"b"^{\phi(f)}}
\end{equation}
such that certain coherence laws are satisfied. The morphism $\phi$ is called \emph{strict} if $\phi(X)f^*=f^*\phi(Y)$ and 
$\phi(f)$ is the identity natural transformation for every $f$.

A \emph{$2$-morphism} $\tau\colon \phi \rr \phi'$ between $1$-morphisms of prederivators is a \emph{modification} of pseudo-natural transformations.
This is defined by a collection of suitably compatible natural transformations in $\cat$
$$\xymatrix@C=40pt{\der(X)\ar@/^15pt/[r]^{\phi(X)}_<(.3){}="a" \ar@/_15pt/[r]_{\phi'(X)}^<(.3){}="b"&\der'(X)\ar@{=>}"a";"b"^{\tau(X)}}$$
for every $X$ in $\Dia$ (see, e.g., \cite[7.5]{borceux1} for the precise definitions).

Let $\PreDer$ (resp.~$\sPreDer$) denote the resulting $2$-category of prederivators, morphisms (resp.~strict morphisms) and $2$-morphisms.
This is an example of a $2$-category formed by $2$-functors, pseudo-natural (or $2$-natural) transformations and modifications (see \cite[Propositions 7.3.3 and 7.5.4]{borceux1}).
The notion of \emph{equivalence} of prederivators is defined in the usual way in terms of the $2$-categorical structure of $\PreDer$. 
Equivalently, a morphism $\phi \colon \der \to \der'$ is an equivalence if and only if $\phi(X)\colon \der(X) \to \der'(X)$ is an equivalence of categories for every 
$X$ in $\Dia$. We also consider the $1$-full sub-$2$-categories $\eqPreDer$ and $\eqsPreDer$ of $\PreDer$ and $\sPreDer$ respectively, which have the same objects and $1$-morphisms 
but whose $2$-morphisms are the invertible modifications. These are categories enriched in groupoids.

\begin{rem}\label{yoneda}
A basic example of a prederivator is the representable prederivator defined by a small category $X$:
$$\cat(-,X)\colon\Dia^{\op}\To\cat.$$
This construction yields a $2$-categorical Yoneda functor
$$\cat\To \sPreDer,$$
which is $1$- and $2$-fully faithful when restricted to $\Dia$.  If we restrict to the $1$-full sub-$2$-category 
whose $2$-morphisms are the invertible natural transformations, we obtain a $2$-functor to $\eqsPreDer$.
\end{rem}

Let $e$ denote the final category with one object $e$ and one morphism $\id{e}$. Given a small category $X$, 
there is a canonical isomorphism of categories $i_{X,-}\colon X\cong\cat(e,X)$ defined as follows. 
An object $x\in \ob X$ defines a functor $i_{X,x}\colon e\r X$ with $i_{X,x}(e)=x$, 
and a morphism $g\colon x\r x'$ in $X$ induces a natural transformation 
$i_{X,g}\colon i_{X,x}\rr i_{X,x'}$ with $i_{X,g}(e)=g$.

Let $X$ be a category in $\Dia$. For every prederivator $\der$ there is a functor
\begin{equation}\label{diaxe}
\dia_{X,e}\colon\der(X)\To \cat(X,\der(e))
\end{equation}
which sends an object $F$ in $\der(X)$ to the functor $\dia_{X,e}(F)\colon X\r \der(e)$ defined by
\begin{align*}
\dia_{X,e}(F)(x)&=i_{X,x}^*F,&
\dia_{X,e}(F)(g\colon x\r x')&=i_{X,g}^*F,
\end{align*}
and a morphism $\varphi\colon F\r G$ in $\der(X)$ to $\dia_{X,e}(\varphi)\colon \dia_{X,e}(F)\rr \dia_{X,e}(G)$, the natural transformation given by
\begin{align*}
\dia_{X,e}(\varphi)(x)&=i_{X,x}^*\varphi.
\end{align*}
This suggests a useful analogy, namely, to regard $\der(e)$ as the underlying
category of $\der$ and the elements of $\der(X)$ as $X$-indexed diagrams in $\der$. We will often write $F_x$ for 
$i_{X,x}^{*}(F)$.

\begin{rem} \label{underlying-diagram}
The functors \eqref{diaxe} assemble to a morphism of prederivators
$$\dia_{-,e}\colon\der\To \cat(-,\der(e))$$
which is the unit of the $2$-adjoint pair
$$\xymatrix@C=60pt{\sPreDer\ar@<.5ex>[r]^-{\text{evaluation at }e}&
\cat.\ar@<.5ex>[l]^-{\text{$2$-Yoneda}}}$$
The counit is the natural isomorphism $\cat(e,X)\cong X$ described above.
\end{rem}

The product of $2$-categories is $2$-functorial, hence for any $Y$ in $\Dia$ and any prederivator 
$\der$, we obtain a new prederivator 
$$\der_Y: = \der(- \times Y) \colon \Dia^{\op} \To \cat.$$ 
The morphism $\dia_{X,e}$ for this new 
prederivator will be denoted by
\begin{equation}\label{diaxy}
\dia_{X,Y} \colon \der(X \times Y)\To\cat(X,\der(Y)).
\end{equation}
Here we use the obvious isomorphism $e\times Y\cong Y$ as an identification. This functor sends an object 
$F$ in $\der(X \times Y)$ to the functor $\dia_{X,Y}\colon X\r\der(Y)$
 defined by
\begin{align*}
\dia_{X,Y}(F)(x)&=(i_{X,x} \times Y)^*F,&
\dia_{X,Y}(F)(g\colon x\r x')&=(i_{X,g} \times Y)^*F,
\end{align*}
and a morphism $\varphi\colon F\r G$ in $\der(X \times Y)$ to $\dia_{X,Y}(\varphi)\colon \dia_{X,Y}(F)\rr \dia_{X,Y}(G)$,  
the natural transformation given by
\begin{align*}
\dia_{X,Y}(\varphi)(x)&=(i_{X,x} \times Y)^*\varphi.
\end{align*}
The functor $\dia_{X,Y}$ may be viewed as taking an $(X \times Y)$-indexed diagram to the underlying
$X$-diagram of $Y$-indexed diagrams in $\der$.

\subsection{Derivators}

A (\emph{right} or \emph{left, pointed, stable/triangulated}) \emph{derivator} is a prederivator that satisfies certain additional properties.
We only briefly review the definitions here. With the exception of Appendix \ref{appA}, we will mainly be concerned with the case of
pointed right derivators.

A \emph{right} \emph{derivator} is a prederivator $\der$ satisfying the following
properties:
\renewcommand{\theenumi}{(Der\arabic{enumi})}
\renewcommand{\labelenumi}{\theenumi}
\begin{enumerate}
\setcounter{enumi}{0}
\setlength{\itemsep}{.2cm}
\item For every pair of small categories $X$ and $Y$ in $\Dia$, the functor induced by the inclusions of the factors to the
coproduct $X\sqcup Y$,
$$\der(X\sqcup Y)\To \der(X)\times \der(Y),$$
is an equivalence of categories. Moreover, $\der(\varnothing)$ is the final category $e$.

\item For every small category $X$ in $\Dia$, the functor
$$(i_{X,x}^{*})_{x \in \ob{X}}  \colon \der(X) \To \prod_{x\in \ob{X}} \der(e)$$
reflects isomorphisms.

\item For every morphism $f\colon X\r Y$ in $\Dia$, the inverse image $f^{*}\colon\der(Y)\r \der(X)$ admits a left adjoint
$f_{!}\colon \der(X)\To \der(Y)$. 

\item Given $f\colon X\r Y$ in $\Dia$ and $y$ an object of $Y$, consider the following diagram in $\Dia$,
$$\xymatrix{f\downarrow y\ar[d]^-{\;}="b"_-{p_{f\downarrow y}}\ar[r]^-{j_{f,y}}_-{\;}="a"&X\ar[d]^-{f}\\
e\ar[r]_-{i_{Y,y}}&Y\ar@{=>}"a";"b"^-{\alpha_{f,y}}}$$
Here $f\downarrow y$ is the comma category whose objects $(x,f(x)\to y)$ are pairs given by an object $x$ in $X$ and a map 
$f(x)\to y$ in $Y$, $j_{f,y}$ is the functor $j_{f,y}(x,f(x)\r y)=x$, and $\alpha_{f,y}(x,f(x)\r y)= (f(x)\r y)$.
Then the diagram obtained by applying $\der$ satisfies the Beck--Chevalley condition, i.e.~the mate
natural transformation
$$c_{f,y}\colon (p_{f\downarrow y})_{!}j^{*}_{f,y}\Longrightarrow i^{*}_{Y,y}f_{!},$$
which is the adjoint of
$$\xymatrix{j^{*}_{f,y}\ar@{=>}[r]^-{\begin{array}{c}
\scriptstyle\text{unit of}\\[-5pt]
\scriptstyle f_!\dashv f^{*}
\end{array}}& j_{f,y}^*f^*f_{!}=(fj_{f,y})^*f_{!}\ar@{=>}[r]^-{\alpha_{f,y}^*f_{!}}& (i_{Y,y}p_{f\downarrow y})^*f_{!}=p_{f\downarrow y}^*i^{*}_{Y,y}f_{!},}$$
is a natural isomorphism. 
\end{enumerate} \

A \emph{left} 
\emph{derivator} $\der$ is a prederivator whose \emph{opposite prederivator} $\der^{\op}$, defined by $\der^{\op}(X)=\der(X^{\op})^{\op}$, is a right derivator. 
A prederivator which is both a left and a right derivator is simply called \emph{derivator}. There is yet another axiom that a prederivator may satisfy, 
\begin{enumerate}
\item[(Der5)] For every pair of small categories $X$ and $\mathcal{I}$ in $\Dia$ where $\mathcal{I}$ is a free finite category, the
canonical functor
$$\dia_{\mathcal{I},X} \colon \der(\mathcal{I} \times X)\To\cat(\mathcal{I},\der(X))$$
is full and essentially surjective.
\end{enumerate}
The inclusion of this axiom in the definition of derivator is somewhat controversial in the literature.  Heller \cite{hellerht} includes (Der5) as part of the
definition. Other authors, e.g.~see \cite{derivateurs(malt), ktdt,franke-ASS,derivators(groth)}, prefer either to
omit it and reserve it for an additional `strongness' property of a derivator, or to replace it with the seemingly weaker version in which
$\mathcal{I}=[1]$. The inclusion of (Der5) matters very little for our purposes here, but we choose to exclude it from the basic
definition.

We recall that a small category is called \emph{pointed} if it has a zero object. A functor between pointed categories
is called pointed if it preserves zero objects. A prederivator $\der$ is called \emph{pointed} if $\der(X)$ is a pointed
category and $f^*\colon \der(Y) \to \der(X)$ is a pointed functor for all $X$ and $f\colon X \to Y$ in $\Dia$. This definition follows
Groth \cite{derivators(groth)}, who showed that it is equivalent for derivators to the original definition, see e.g.~\cite{ktdt}.

We recall the definition of cocartesian squares for right derivators. Consider the `commutative square' category $\square= [1] \times [1]$.
A \emph{commutative square} in a prederivator $\der$ is an object $F$ of $\der(\square)$. There is a subcategory
$\ulcorner \subseteq \square$ 
which can be depicted as follows,
$$\begin{array}{c}
\xymatrix{
(0,0)\ar[r]\ar[d] \ar@{}[rd]|{\displaystyle \square}&(1,0)\ar[d]\\
(0,1)\ar[r]&(1,1)
}\end{array},\quad
\begin{array}{c}
\xymatrix{
(0,0)\ar[r]\ar[d] \ar@{}[rd]|{\displaystyle \ulcorner}&(1,0)\\
(0,1)&
}\end{array}
.$$
Denote the inclusion functor by
\begin{align*}
i_{\ulcorner}\colon \ulcorner&\To \square
.
\end{align*}
If $\der$ is a right derivator, a commutative square $F$ in $\der$ is called \emph{cocartesian} if the counit
$$(i_{\ulcorner})_{!}i_{\ulcorner}^{*}F\To F$$ is an isomorphism. If $\der$ is a left derivator,
a commutative square $F$ in $\der$ is called \emph{cartesian} if it is cocartesian in $\der^{\op}$.
A pointed derivator which satisfies (Der5) is called \emph{stable} (or \emph{triangulated}) if
cocartesian and cartesian squares coincide.

A crucial point to note regarding all of the above definitions is that only the notion of a prederivator constitutes
\emph{structure}, while the additional axioms assert \emph{properties}. We also emphasize that the property of being a right
(resp.~left, pointed, stable/triangulated) derivator is invariant under equivalences of prederivators. Moreover, if $\der$ is 
a pointed right derivator, then so is $\der_Y$ for every $Y$ in $\Dia$. 

A morphism of right derivators $\phi\colon \der \to \der'$ is called \emph{cocontinuous} if for every $f\colon X \to Y$ in $\Dia$, the
canonical natural transformation $$f_! \phi(X) \Longrightarrow \phi(Y) f_{!},$$
adjoint to
$$\xymatrix@C=40pt{\phi(X)\ar@{=>}[r]^-{\begin{array}{c}
\scriptstyle\text{unit of}\\[-5pt]
\scriptstyle f_!\dashv f^{*}
\end{array}}& \phi(X)f^*f_{!}\ar@{=>}[r]^-{\phi(f)^{-1}f_{!}}& f^*\phi(Y)f_{!},}$$
is an isomorphism\footnote{This condition
is comparable to \emph{right} exactness of a functor. This justifies the term \emph{right} derivator.}.
An easy application of (Der1) shows that the components of a cocontinuous morphism between pointed right derivators are
automatically pointed functors.

Let $\Der$ (resp.~$\sDer$) denote the $2$-full sub-$2$-category of $\PreDer$ (resp.~$\sPreDer$) given by the pointed right
derivators, cocontinuous (strict) morphisms and $2$-morphisms. Let $\eqDer$ and $\eqsDer$ denote the $1$-full sub-$2$-categories
of $\Der$ and $\sDer$ whose $2$-morphisms are the invertible modifications.

\subsection{Examples}\label{examples}
The examples of prederivators that we are interested in arise from categories with weak equivalences as follows. Let $(\mathcal{C}, \mathcal{W})$
be a pair consisting of a small category $\mathcal{C}$ together with a subcategory $\mathcal{W}$ which contains the isomorphisms. The morphisms of
$\mathcal{W}$ are called \emph{weak equivalences}. The  \emph{homotopy category} of $(\mathcal{C}, \mathcal{W})$  is the localization
$$\ho\mathcal C:=\mathcal C[\mathcal W^{-1}].$$

For every object $X$ in $\Dia$, the diagram category $\mathcal{C}^{X}$ together with the subcategory of
objectwise weak equivalences of functors is again a category with weak equivalences 
 $(\mathcal{C}^X, \mathcal{W}^{X})$. The
choice of objectwise weak equivalences is natural in $X$, so there is a prederivator
$\der(\mathcal{C}, \mathcal{W}) \colon \Dia^{\op} \to \cat$ given by the homotopy categories of all relevant diagram categories, i.e.~it
is defined on objects by
$$\der(\mathcal{C}, \mathcal{W})(X):=  \ho (\mathcal{C}^{X}) 
$$
and on $1$- an $2$-morphisms in the canonically induced way. 

A functor $F\colon \mathcal{C} \to \mathcal{C}'$ that preserves the
weak equivalences $F(\mathcal{W}) \subset \mathcal{W}'$ induces a (strict) morphism of prederivators $\der(F) \colon \der(\mathcal{C},
\mathcal{W}) \to \der(\mathcal{C}', \mathcal{W}')$. Such functors are called \emph{homotopical}. A homotopical functor $F\colon \mathcal{C} \to \mathcal{C}'$
is a \emph{derived equivalence} if it induces an equivalence of homotopy categories $\ho{F}\colon \ho{\mathcal{C}} \stackrel{\simeq}{\to} \ho{\mathcal{C}'}$. 
A natural transformation $\alpha\colon F \Rightarrow F'$ of homotopical functors $F,F'\colon (\C, \mathcal{W}) \to (\C', \mathcal{W}')$ defines a $2$-morphism $\der(\alpha)\colon \der(F) \Rightarrow \der(F')$ in $\sPreDer$. If the components of the natural
transformation $\alpha$ are given by weak equivalences, then $\der(\alpha)$ is in $\eqsPreDer$. 

We note
that if  $\mathcal{W}$ is the subcategory of isomorphisms, $\der(\C, \mathcal{W})=\cat(-,\mathcal C)$ is the representable
prederivator of Remark \ref{yoneda}. We will normally write $\der(\mathcal{C})$ when the choice of $\mathcal{W}$ is clear 
from the context.

For well-behaved categories with weak equivalences $(\mathcal{C}, \mathcal{W})$, the associated prederivator $\der(\mathcal{C})$ is
a (right or left, pointed, stable/triangulated) derivator. We refer the reader to Cisinski \cite{ciscd} for a systematic treatment of
the results in this direction. Here we will be particularly concerned with categories with weak equivalences that arise from Waldhausen
categories \cite{akts}. Following Cisinski \cite{ikted}, we say that a Waldhausen category
$(\mathcal{C}, co \mathcal{C}, w \mathcal{C})$ is \emph{derivable} if it satisfies the `$2$-out-of-3' axiom and every morphism in
$\mathcal{C}$ can be written as the composition of a cofibration followed by a weak equivalence. The following theorem is a
special case of results proved in \cite{ciscd,ikted}.

\begin{thm}[Cisinski] \label{cisinski} \
\begin{itemize}
\item[(a)] Let $(\mathcal{C}, co \mathcal{C}, w \mathcal{C})$ be a derivable Waldhausen category. Then the associated 
prederivator $\der(\mathcal{C}): \dirf^{\op} \to \cat$ is a pointed right derivator which also satisfies (Der5).
\item[(b)] An exact functor of derivable Waldhausen categories
$F\colon (\mathcal{C}, co \mathcal{C}, w \mathcal{C}) \to (\mathcal{C}', co \mathcal{C}', w \mathcal{C}')$ induces a cocontinuous
morphism $\der(F)\colon \der(\mathcal{C}) \to \der(\mathcal{C}')$ in $\sDer$.
\item[(c)] Moreover, the morphism $\der(F) \colon \der(\C) \to \der(\C')$ is an equivalence in $\sDer$ if and only if
$\ho{F} \colon \ho{\C} \to \ho{\C'}$ is an equivalence of categories.
\end{itemize}
\end{thm}

A derivable Waldhausen category $(\mathcal{C}, co \mathcal{C}, w \mathcal{C})$ is called \emph{strongly saturated} if it satisfies the property that
a morphism in $\mathcal{C}$ is a weak equivalence if and only if it becomes an isomorphism in the homotopy category. A derivable Waldhausen category 
with functorial factorizations is strongly saturated if and only if the weak equivalences are closed under retracts (see \cite[Theorems 5.5 and 6.4]{aKtaht}). 

\section{Simplicial enrichments of (pre)derivators}\numberwithin{equation}{subsection}

We recall that the \emph{simplex category} $\Delta$ consists of the finite ordinals $[n]=\{0<\cdots<n\}$, $n\geq 0$, and 
the non-decreasing maps between them. Thus, it is contained in $\dirf$, and in fact also in any other possible category of 
diagrams $\Dia$. The naturality of the construction $Y \mapsto \der_Y$ shows that there is a $2$-functor
$$\Dia^{\op} \times \sPreDer \To \sPreDer$$
which may be regarded as a `co-tensor $2$-structure' of $\sPreDer$ over $\Dia$. Using this, we can
associate to every prederivator $\der$ a simplicial object 
$\der_{\bullet}$ in $\sPreDer$ with 
$$\der_n=\der([n]\times -).$$ 
In particular, we have $\der_0= \der$. Faces and degeneracies are morphisms of prederivators since both $\der$ and the product 
of $2$-categories are $2$-functorial. This natural simplicial object will be used to define an enrichment of the underlying category 
of $\sPreDer$ over simplicial sets.

\subsection{Definition of $\PREDER$} \label{3.1} We define a simplicially enriched category $\PREDER$ with prederivators 
as objects and morphism simplicial sets 
$$\PREDER(\der, \der')_{\bullet} = \ob{\sPreDer(\der, \der'_{\bullet})}.$$
The composition is defined by simplicial maps
$$\PREDER(\der', \der'')_{\bullet} \times \PREDER(\der, \der')_{\bullet} \To \PREDER(\der, \der'')_{\bullet}$$
which send pairs of strict morphisms  $\phi\colon \der \to \der'([n]\times-)$ and $\psi\colon \der' \to \der''([n]\times-)$ to the composite
$$\xymatrix{\der \ar[r]^-{\phi}& \der'([n]\times-) \ar[rr]^-{\psi([n]\times-)}&& \der''([n] \times [n]\times -) \ar[rr]^-{\der''(\triangle\times-)}&& \der''([n]\times-),}$$
where $\triangle\colon [n] \to [n] \times [n]$ is the diagonal functor.

To see that the composition is associative, consider strict morphisms as follows
$$\phi\colon \der \To \der'([n]\times-)$$ 
$$\psi\colon \der' \To \der''([n]\times-)$$
$$\xi\colon \der'' \To \der'''([n]\times-).$$
Then it suffices to show that the leftmost and rightmost morphisms in the following diagram coincide
$$
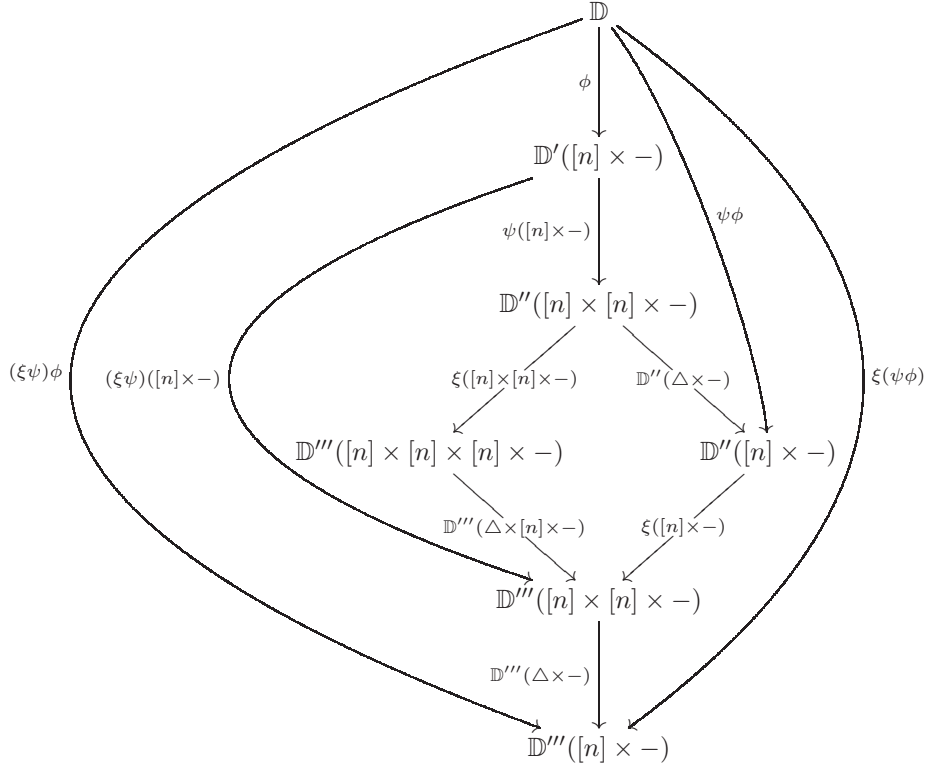
\xymatrix@!C=40pt@R=40pt{
&\der\ar[d]_-\phi\ar@(rd,u)[dddr]^-{\psi\phi}\ar@/^100pt/[ddddd]^-{\xi(\psi\phi)}
\ar@/_200pt/[ddddd]_-{(\xi\psi)\phi}&\\
&\der'([n]\times-)\ar[d]_{\psi([n]\times-)}\ar@/_140pt/[ddd]_-{(\xi\psi)([n]\times-)}&\\
&\der''([n]\times[n]\times-)\ar[rd]|-{\der''(\triangle\times-)}\ar[ld]|-{\xi([n]\times[n]\times-)}&\\
\der'''([n]\times[n]\times[n]\times-)\ar[rd]|-{\der'''(\triangle\times[n]\times-)}&&\der''([n]\times-)\ar[ld]|{\xi([n]\times-)}\\
&\der'''([n]\times[n]\times-)\ar[d]_-{\der'''(\triangle\times-)}&\\
&\der'''([n]\times-)&
}$$
All cells in this diagram commute by definition, except for the inner square. If the inner square were commutative, 
the result would follow immediately. However, the post-composition of the square with $\der'''(\triangle\times-)$ yields a
commutative square, and this suffices. Indeed, since the diagonal functor is coassociative, 
$(\triangle\times [n])\triangle=([n]\times\triangle)\triangle$, it is enough to show that the slightly different square
$$
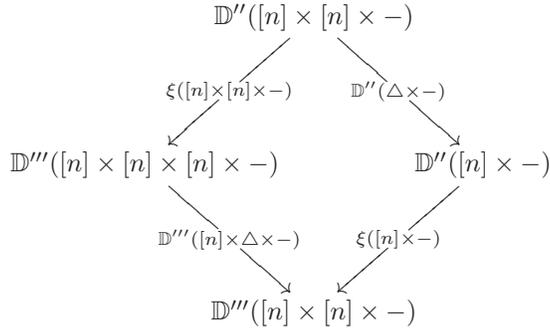
\xymatrix@!C=40pt@R=40pt{
&\der''([n]\times[n]\times-)\ar[rd]|-{\der''(\triangle\times-)}\ar[ld]|-{\xi([n]\times[n]\times-)}&\\
\der'''([n]\times[n]\times[n]\times-)\ar[rd]|-{\der'''([n]\times\triangle\times -)}&&\der''([n]\times-)\ar[ld]|{\xi([n]\times-)}\\
&\der'''([n]\times[n]\times-)&
}$$
commutes. Note that the only difference between this last square and the inner square in the previous diagram is in the arrow 
$\searrow$. This last square commutes because $\xi$ is a strict morphism.

This simplicial enrichment can be used to introduce homotopy theoretic notions in the world prederivators but these will be too
coarse for our purposes here. For a more appropriate notion of homotopy, we consider the subobject of $\der_{\bullet}$ defined by ``simplicially constant''
diagrams. 

\subsection{Definition of $\eqPREDER$} Given a prederivator $\der$ and $Y$ in $\Dia$, there is a prederivator 
$\der(Y \times-)_{\eq}$ equipped with a strict morphism $i_{\eq}\colon \der(Y \times -)_{\eq}\r\der(Y \times -)$ 
such that for all $X$ in $\Dia$, 
$$i_{\eq}(Y)\colon \der(Y \times X)_{\eq}\To\der(Y \times X)$$ 
is the inclusion of the full subcategory spanned by the objects $F$ such that the underlying $Y$-diagram
$$\dia_{Y,X}(F)\colon Y \longrightarrow \der(X)$$ 
sends each morphism of $Y$ to an isomorphism in $\der(X)$.

To show that this is well-defined, it is enough to check that given $f\colon X \r Z$ in $\Dia$ and $F$ in 
$\der(Y \times Z)_{\eq}$, the object $(Y \times f)^{*}(F)$ in $\der(Y \times X)$ is actually in 
$\der(Y \times X)_{\eq}$. Let $g \colon y \r y'$ be a morphism in $Y$. We have
\begin{align*}
(i_{Y,g} \times X)^{*}(Y \times f)^{*}(F)&= ((Y \times f)(i_{Y,g} \times X))^{*}(F)\\
&=(i_{Y,g} \times f)^{*}(F)\\
&=((i_{Y,g} \times Z)(Y \times f))^{*}(F)\\
&=(e \times f)^{*}(i_{Y,g} \times Z)^{*}(F).
\end{align*}
Since $F$ is in $\der(Y\times Z)_{\eq}$,  $(i_{Y,g}\times Z)^{*}(F)$ is an isomorphism, hence so is 
$(i_{Y,g}\times X)^{*}(Y \times f)^{*}(F)$ for any morphism $g$ in $Y$, therefore $(Y \times f)^{*}(F)$  
is in $\der(Y\times X)_{\eq}$. (See also Remark \ref{underlying-diagram}.) 

\
\

Hence, for any prederivator $\der$, there is a 
simplicial prederivator $\der_{\eq,\bullet}$ with 
$\der_{\eq,n}=\der([n]\times-)_{\eq}$ equipped with a 
morphism $i_{\eq}\colon \der_{\eq,\bullet}\r\der_{\bullet}$ of simplicial prederivators. Note that $\der_{\eq,0}= \der$.

We define a simplicially enriched category $\eqPREDER$ with prederivators as objects and morphism simplicial sets 
$$\eqPREDER(\der, \der')_{\bullet} = \ob{\sPreDer(\der, \der'_{\eq,\bullet})}$$
such that the morphisms of simplicial prederivators $i_{\eq}\colon \der'_{\eq,\bullet}\r\der'_{\bullet}$
induce a simplicial functor $i_{\eq}\colon\eqPREDER\r\PREDER$.

To show that this is a well-defined simplicial subcategory, we check that the composition in $\PREDER$ of two composable $n$-simplices in $\eqPREDER$ is again in $\eqPREDER$, 
i.e.~given strict morphisms  $\phi\colon \der \to \der'([n]\times-)_{\eq}$ and $\psi\colon \der' \to \der''([n]\times-)_{\eq}$, 
we show that the composite
$$\xymatrix{\der \ar[r]^-{\phi}& \der'([n]\times-) \ar[rr]^-{\psi([n]\times-)}&& \der''([n] \times [n]\times -) \ar[rr]^-{(\triangle\times-)^{*}}&& \der''([n]\times-)}$$
takes values in $\der''([n]\times-)_{\eq}$. Given an object $F$ in $\der (X)$ and a morphism $g\colon x\r x'$ in $[n]$, 
consider the following diagram
$$\xymatrix@C=20pt@R=40pt{\der(X) \ar[r]^-{\phi(X)}& \der'([n]\times X) \ar[rr]^-{\psi([n]\times X)}
\ar@/^23pt/[d]^<(.55){(i_{[n],x'}\times X)^{*}}_<(.6){}="b"
\ar@/_23pt/[d]_{(i_{[n],x}\times X)^{*}}^<(.6){}="a"
&& \der''([n] \times [n]\times X) \ar[rr]^-{(\triangle\times X)^{*}}
\ar@/^30pt/[d]^<(.25){([n]\times i_{[n],x'}\times X)^{*}}_<(.6){}="d"
\ar@/_30pt/[d]_<(.25){([n]\times i_{[n],x}\times X)^{*}}^<(.6){}="c"
&& \der''([n]\times X)\ar@/^22pt/[dd]|<(.7){(i_{[n],x'}\times X)^{*}}_<(.5){}="h"
\ar@/_22pt/[dd]|<(.2){(i_{[n],x}\times X)^{*}}^<(.5){}="g"
\\
&\der'(X) \ar[rr]_-{\psi( X)}&&\der''([n]\times X)
\ar@/^23pt/[d]^{(i_{[n],x'}\times X)^{*}}_<(.6){}="f"
\ar@/_23pt/[d]_{(i_{[n],x}\times X)^{*}}^<(.6){}="e"
&\\
&&&\der''(X)\ar@{=}[rr]&&\der''(X)
\ar@{=>}"a";"b"^-{(i_{[n],g}\times X)^{*}}
\ar@{=>}"c";"d"^-{([n]\times i_{[n],g}\times X)^{*}}
\ar@{=>}"e";"f"^-{(i_{[n],g}\times X)^{*}}
\ar@{=>}"g";"h"^-{(i_{[n],g}\times X)^{*}}
}$$
This diagram satisfies several commutativity properties. The subdiagram of functors formed by the straight arrows and the arrows which are curved to the left (resp.~right) is commutative. In the middle square, the natural transformations $\psi(X)(i_{[n],g}\times X)^{*}=([n]\times i_{[n],g}\times X)^{*}\psi([n]\times X)$ coincide, since $\psi$ is a strict morphism. In the rightmost region, the two horizontally composable natural transformations compose to $(i_{[n],g}\times X)^{*}(\triangle\times X)^{*}$, since $\der''$ is a $2$-functor.

Since $\phi$ takes values in $\der'([n]\times-)_{\eq}$, we have that $(i_{[n],g}\times X)^{*}\phi(X)$ is a natural isomorphism. 
Moreover, $\psi$ takes values in $\der''([n]\times-)_{\eq}$, therefore $(i_{[n],g}\times X)^{*}\psi(X)$ is also a natural 
isomorphism. This, together with the aforementioned commutativity properties, shows that 
$(i_{[n],g}\times X)^{*}(\triangle\times X)^{*}\psi([n]\times X)\phi(X)(F)$ is an isomorphism.
\

\

The passage from $\PREDER$ to $\eqPREDER$ is reminiscent of the passage from the category of $\infty$-categories, regarded
as an $(\infty,2)$-category, to the associated $\infty$-category defined by restriction to the maximal $\infty$-groupoids
of the morphism $\infty$-categories. More on the viewpoint that regards well-behaved types of prederivators as models for 
homotopy theories will be discussed in Appendix \ref{appA}, see also \cite{pcthqd}.

\subsection{Strong equivalences} The prederivator
$\der([1]\times-)_{\eq}$ together
with the following  factorization of the diagonal natural transformation
$$\der \stackrel{s _0}{\longrightarrow} \der([1]\times-)_{\eq} \xrightarrow{(\partial _1, \partial _0)} \der \times \der$$
will be regarded as a path object associated with $\der$. We can now introduce some basic homotopical notions in the context of 
prederivators.

\begin{defn} Let $\phi_0, \phi_1\colon \der \to \der'$ be two strict morphisms of prederivators. A \emph{strong isomorphism 
from} $\phi_0$ \emph{to} $\phi_1$ is a $1$-simplex of $\eqPREDER(\der, \der')$,
$$\Psi\colon \der \To \der'([1]\times-)_{\eq},$$ such that
$\partial_1(\Psi)= \phi_0$ and $\partial_0(\Psi)=\phi_1$. We say that $\phi_0$ is \emph{strongly isomorphic} to $\phi_1$,
written $\phi_0 {\simeq} \phi_1$, if there is a zigzag of strong isomorphisms from $\phi_0$ to $\phi_1$.
\end{defn}

Obviously the relation \ $\simeq$ \ is exactly the relation that two vertices of $\eqPREDER(\der, \der')$ lie on the same
component.

\begin{defn}
Let $\der$ and $\der'$ be prederivators.
\begin{itemize}
 \item[(a)] A strict morphism $\phi\colon \der \to \der'$ is called a \emph{strong} (or \emph{coherent}) \emph{equivalence} if there is a strict
 morphism $\psi\colon \der' \to \der$ such that $ \id{_{\der}} \simeq \psi \phi$ and
$\phi \psi \simeq \id{_{\der'}}$.
\item[(b)] $\der$ and $\der'$ are called \emph{strongly} (or \emph{coherently}) \emph{equivalent} if there is a strong equivalence
$\phi\colon \der \to \der'$.
\end{itemize}
\end{defn}

\begin{rem}
A strong isomorphism $\Phi$ from $\phi_0$ to $\phi_1$ induces a natural isomorphism $$\dia_{[1],-}(0\r 1)\colon\phi_0 \Longrightarrow \phi_1.$$
From this follows that strong equivalences of prederivators are also equivalences in the $2$-categorical sense of the previous section.
\end{rem}

\begin{exm}\label{path-objects}
For every prederivator $\der$ and any $X$ in $\Dia$ with an initial object $x_{0}\in\Ob X$, the prederivator $\der(X\times-)_{\eq}$ is strongly equivalent to $\der$. Indeed, consider the morphisms 
\begin{align*}
(p\times -)^{*}\colon\der &\To\der(X\times-)_{\eq},&
(i_{X,x_{0}}\times-)^{*}\colon\der(X\times-)_{\eq} &\To\der,&
\end{align*}
where $p\colon X\r e$ is the unique functor in this direction. Note that the underlying $X$-diagrams of elements in the image of 
$(p \times -)^{*}$ are constant functors (cf.~Remark \ref{underlying-diagram}). Since $pi_{X,x_{0}}$ is the identity functor on $e$, we have 
$(i_{X,x_{0}}\times-)^{*}(p\times -)^{*}=\id{\der}$. Moreover, $i_{X,x_{0}}p\colon X\r X$ is the constant functor 
$x\mapsto x_{0}$ and since $x_{0}$ is initial, there is a unique functor $H\colon [1]\times X\r X$ with $H(0,-)=i_{X,x_{0}}p$ 
and $H(1,-)=\id{X}$. The induced functor
$$(H\times-)^{*}\colon\der(X\times-)_{\eq}\To\der(([1]\times X)\times-)_{\eq}=\der(X\times-)_{\eq,1}$$
is a strong isomorphism  from $(p\times -)^{*}(i_{X,x_{0}}\times-)^{*}$ to $\id{\der(X\times-)_{\eq}}$. One can argue similarly if $X$ has a 
final object. This shows, in particular, that the face and degeneracy operators in $\der_{\eq,\bullet}$ are strong equivalences. 
\end{exm}


The notion of strong equivalence is different from the standard notion of equivalence defined in terms of the
$2$-categorical structure of $\PreDer$. This observation will be a crucial in connection with the definitions of 
$K$-theory that follow in the next sections.

\begin{exm} \label{iso-objects}
Let $\der$ be a prederivator and $\iso_n \der$ denote the prederivator for which $(\iso_n \der) (X)$ is the full subcategory spanned
by the strings of $n$ composable isomorphisms in the diagram category of $\cat([n], \der(X))$. Then the canonical `inclusion of identities' 
morphism $\der \to \iso_n \der$ is clearly an equivalence of prederivators, but not a strong equivalence in general. This assertion is 
a consequence of the invariance properties of Waldhausen $K$-theory and will be justified in Remark \ref{coherent-eq-vs-eq} below. 
\end{exm}

\begin{rem} \label{homotopy-eq-coherent-eq}
In connection with the examples of section \ref{examples}, a natural transformation $\alpha\colon F\rr F'$ between 
homotopical functors $F,F'\colon (\C, \mathcal{W}) \to (\C', \mathcal{W}')$ induces 
a $1$-simplex in $\PREDER$
$$\alpha_{*}\colon \der(\mathcal{C}', \mathcal{W}') \longrightarrow \der(\mathcal{C}, \mathcal{W})([1]\times-)$$
which upgrades the $2$-morphism $\der(\alpha)\colon \der(F) \Rightarrow \der(F')$. If $\alpha$ takes values in 
$\mathcal W'$ then $\alpha_{*}$ is a $1$-simplex in $\eqPREDER$. This implies that for every homotopical functor 
$F \colon (\C, \mathcal{W}) \to (\C', \mathcal{W}')$ which admits a `homotopy inverse', i.e.~there is a homotopical functor 
$G \colon \C' \to \C$ such that the composites $FG$ and $GF$ can be connected to the respective identity functors via 
zigzags of natural weak equivalences, the associated morphism $\der(F)$ is a strong equivalence.
\end{rem}

A $2$-category $\mathcal C$ will be regarded as a simplicial category $N_{\bullet}\mathcal C$  via the nerve functor $N_{\bullet} \colon \mathcal{C}at \to \SSet$ 
from small categories to simplicial sets, which preserves products. We have a simplicial functor
$$\PREDER \To N_{\bullet} \sPreDer,$$
which is the identity on objects, and is given on morphisms by the
simplicial maps
\begin{equation}\label{truncation-derivators}
\PREDER(\der, \der') \To N_{\bullet} \sPreDer(\der, \der')
\end{equation}
defined using the functors $\dia_{[\bullet], -}$. These simplicial maps also restrict to  simplicial maps
\begin{equation*}
\eqPREDER(\der, \der') \To N_{\bullet} \eqsPreDer(\der, \der')
\end{equation*}
which assemble to a simplicial functor
$$\rho\colon\eqPREDER \To N_{\bullet} \eqsPreDer$$
given by the identity on objects.

Consider the following adjoint pairs
$$\xymatrix{\SSet\ar@<.5ex>[r]^-{\tau_{1}}&\mathcal{C}at\ar@<.5ex>[r]\ar@<.5ex>[l]^-{N_{\bullet}}&\mathcal{G}rd.\ar@<.5ex>[l]^-{\text{incl.}}}$$
Here $\mathcal{G}rd$ is the category of groupoids, $\mathcal{G}rd\r \mathcal{C}at$ is the inclusion, the lower arrows are 
the right adjoints, $\tau_{1}$ is the fundamental category functor, and the composite 
$\SSet\r \mathcal{G}rd$ is the fundamental groupoid functor, denoted by $\Pi_1$. All these functors preserve products, 
hence, for example, we can apply them to a simplicial category $\mathcal{S}$ to obtain a $2$-category $\tau_1\mathcal{S}$, 
or a category enriched in groupoids $\Pi_1\mathcal{S}$. In particular, the simplicial functors 
\eqref{truncation-derivators} and $\rho$ above also define $2$-functors
$$\tau_1\PREDER \To \sPreDer,\qquad
\Pi_1\eqPREDER \To \eqsPreDer,$$
by adjunction. These functors are not $2$-equivalences of $2$-categories. This means that the simplicial enrichment of the category 
of prederivators encodes more structure that the $2$-category of prederivators.

Similarly let $\DER$ and $\eqDER$ denote the corresponding simplicial subcategories of $\PREDER$ and $\eqPREDER$ respectively.
In both cases the objects are pointed right derivators, and for a pair of pointed right derivators $\der$ and $\der'$ we have:
$$\DER(\der, \der')_{\bullet} = \ob \sDer(\der, \der'_{\bullet})$$
$$\eqDER(\der, \der')_{\bullet} = \ob \sDer(\der, \der'_{\eq,\bullet}).$$
This is well-defined because if $X$ is in $\Dia$ and $\der$ is a right
(resp.~left, pointed, stable/triangulated) derivator, then so are $\der(X\times-)$ and $\der(X\times-)_{\eq}$.

Specializing the discussion above  to pointed right derivators, we define similarly a simplicial functor 
$$\rho\colon \eqDER \To N_{\bullet} \eqsDer.$$
Again, the associated $2$-functor
$$\Pi_1\eqDER \To \eqsDer$$
is not a $2$-equivalence.

\section{Waldhausen $K$-theory of derivators}\numberwithin{equation}{subsection}

In this section, we define the Waldhausen $K$-theory of a pointed right derivator and show that it agrees
with the Waldhausen $K$-theory of a strongly saturated derivable Waldhausen category.

\subsection{The $\S_{\bullet \bullet}$-construction.} First we recall the analogue of Waldhausen's $\S_{\bullet}$-construction 
in the setting of derivators due to Garkusha \cite{sdckt1, sdckt2}. Let $\der$ be a pointed right derivator. We denote by $\Ar[n]$ the category (finite poset) of arrows of the poset $[n]$. Let 
$\S_n \der$ denote the full subcategory of $\der(\Ar[n])$ spanned by objects $F$ that satisfy the following conditions:
\begin{enumerate}
\item[(i)] for every $0 \leq i \leq n$, the object $i_{\Ar[n],i\r i}^{*}F \in \ob{\der(e)}$ is a zero object,
\item[(ii)] for every $0 \leq i \leq j \leq k \leq n$, the restriction of $F$ along the inclusion of the following subcategory of $\Ar[n]$, isomorphic to $\square$,
\[
 \xymatrix{
(i \to j) \ar[r] \ar[d] & (i \to k) \ar[d] \\
(j \to j) \ar[r] & (j \to k)
}
\]
is a cocartesian object of $\der(\square)$.
\end{enumerate}

This defines a simplicial category $\S_{\bullet} \der$ where the simplicial operators are defined by the structure of $\der$ as
a prederivator. Since morphisms in $\sDer$ preserve cocartesian squares, it follows easily that the correspondence
$\der \mapsto \S_{\bullet} \der$ defines a functor from the underlying $1$-category $\eqs1Der$ of $\eqsDer$ (or $\sDer$), which can also be obtained by forgetting the simplices of positive dimension in $\eqDER$, to 
the ($1$-)category of simplicial categories. 

For the definition of Waldhausen $K$-theory, we need to consider a more refined version of this construction. Let 
$\S_{\bullet \bullet} \der$ be the bisimplicial set whose set of $(n,m)$-simplices $\S_{n,m} \der$
is the set of objects
$$F \in \ob \der([m]\times\Ar[n])_{\eq}$$
such that:
\begin{itemize}
\item[($\ast$)] for every $j\colon [0] \to [m]$ the object $(j\times\Ar[n])^*(F) \in \Ob \der(\Ar[n])$ is in $\S_n \der$.
\end{itemize}
Note that if this condition holds for \emph{some} $j \colon [0] \to [m]$ then it holds for all $j$. The bisimplicial operators are 
again defined using the structure of the underlying prederivator. Moreover, it is easy to see that the construction is natural in $\der$,
that is, we obtain a functor $\der \mapsto \S_{\bullet \bullet} \der$ from the underlying $1$-category $\eqs1Der$ 
of $\eqsDer$ to the category of bisimplicial sets.

\begin{defn}
The Waldhausen $K$-theory of a pointed right derivator $\der$ is defined to be the space
$\WK(\der) := \Omega | \S_{\bullet \bullet} \der|.$
\end{defn}

Our next goal is to show that the functor $\WK$ can be extended to a simplicial functor from $\eqDER$ to the 
simplicially enriched category of topological spaces $\TOP$. Here the $n$-simplices of the simplicial mapping 
space $\TOP(X,Y)$ between topological spaces $X$ and $Y$ are the continuous maps $X \times \Delta^n \to Y$ 
where $\Delta^n$ denotes the geometric $n$-simplex. Since both the geometric realization functor and the loop 
space functor are simplicial,
it is enough to show that the functor
$$\der\mapsto \diag\S_{\bullet \bullet}\der$$ 
can be extended to a simplicial functor from $\eqDER$ to the standard simplicially enriched category of simplicial sets $\SSET$. 
We recall that for simplicial sets $X$ and $Y$, the $n$-simplices of $\SSET(X, Y)$ 
are the simplicial maps $X \times \Delta[n] \to Y$ where $\Delta[n]$ denotes the $n$-simplex and $|\Delta[n]| \cong \Delta^n$. 
A useful way of describing an $n$-simplex of $\SSET(X,Y)$ is by giving a natural transformation as follows (cf.~\cite[1.4]{akts})
$$\xymatrix@C=30pt@R=8pt{&\Delta^{\op}\ar@/^8pt/[rd]^{X}\ar@{=>}[dd]_{\alpha}&
\\(\Delta\downarrow[n])^{\op}\ar@/^8pt/[ru]^{\text{source}}\ar@/_8pt/[rd]_{\text{source}}&&\operatorname{\mathtt{Set}},\\
&\Delta^{\op}\ar@/_8pt/[ru]_{Y}&}$$
Such a natural transformation $\alpha$ produces a simplicial map $\phi \colon X \times \Delta[n] \to Y$ which 
is defined by 
$$\phi(x, [k] \stackrel{\sigma}{\to} [n]) = \alpha(\sigma)(x).$$ 
Conversely, a simplicial map $\phi \colon X \times \Delta[n] \to Y$ defines the components of such a natural transformation 
by setting $\alpha([k] \stackrel{\sigma}{\to} [n]) = \phi(-, \sigma): X_k \to Y_k$. 

\begin{prop} \label{$K$-theory-as-a-simplicial-fun}
Waldhausen $K$-theory extends to a simplicial functor 
$$\WK \colon \eqDER \To \TOP.$$
\end{prop}
\begin{proof}
As remarked above, it suffices to show that the ($1$-)functor 
\begin{align*}
\eqs1Der &\To \mathtt{SSet},\\
\der&\; \mapsto \;\diag \S_{\bullet \bullet} \der,
\end{align*}
extends to a simplicial functor from $\eqDER$ to $\SSET$. This extension is defined as follows: given 
an $n$-simplex $\phi \colon \der \to \der'([n]\times-)_{\eq}$ in $\eqDER$, its image in $\SSET$
is an $n$-simplex which we specify by giving the associated natural transformation 
$$\xymatrix@C=30pt@R=8pt{&\Delta^{\op}\ar@/^8pt/[rd]^{\diag \S_{\bullet \bullet} \der}\ar@{=>}[dd]_{\phi_{*}}&
\\(\Delta\downarrow[n])^{\op}\ar@/^8pt/[ru]^{\text{source}}\ar@/_8pt/[rd]_{\text{source}}&&\operatorname{\mathtt{Set}},\\
&\Delta^{\op}\ar@/_8pt/[ru]_{\diag \S_{\bullet \bullet} \der'}&}$$
The component of $\phi_*$ at an object $\sigma \colon [k] \to [n]$ in $\Delta\downarrow[n]$ is 
the map $\phi_{*}(\sigma)\colon \S_{k,k}\der\r\S_{k,k}\der'$ defined as the (co)restriction of the map on objects that comes 
from the following functor:
$$\xymatrix{
 \der([k]\times\Ar[k])_{\eq}\ar[d]^{\phi([k]\times\Ar[k])}\\ \der'(([n]\times[k])\times\Ar[k])_{\eq}\ar[d]^{\der'(\sigma \times [k]\times\Ar[k])}\\ \der'(([k]\times[k])\times\Ar[k]))_{\eq}\ar[d]^{\der'(\triangle\times\Ar[k])}\\
  \der'([k]\times\Ar[k])_{\eq}}$$
It is straightforward to check that $\phi_{*}$ is a natural transformation. Moreover, it is easy to check that the 
correspondence $\phi \mapsto \phi_*$ respects the composition (by arguments analogous to those in \ref{3.1}).
\end{proof}

As an immediate consequence, we have the following invariance property of Waldhausen $K$-theory.

\begin{cor}
Let $\phi\colon \der \to \der'$ be a strong equivalence of pointed right derivators. Then the
induced map $\WK(\phi)\colon \WK(\der) \to \WK(\der')$ is a homotopy equivalence.
\end{cor}

\subsection{The $\s_{\bullet}$-construction} We mention a variant of the $\S_{\bullet \bullet}$-construction which 
is actually the analogue of Waldhausen's $\s_\bullet$-construction in this context (cf.~\cite[1.4]{akts}).
Let $\s_\bullet \der$ denote the simplicial set with $n$-simplices
$$\s_n \der := \ob{\S_n \der} = \S_{n, 0} \der$$
and define
$$\wk(\der):= \Omega | \s_{\bullet} \der |.$$
The inclusion of the $0$-simplices defines a canonical comparison map $$\iota \colon \wk(\der) \To \WK (\der).$$

\begin{prop} \label{s-vs-S}
The comparison map $\iota$ is a weak equivalence.
\end{prop}

Proposition \ref{s-vs-S} is a consequence of the following lemma (cf.~\cite[Lemma 1.4.1]{akts}).

\begin{lem} \label{s-construction}
Let $\phi, \phi'\colon \der \to \der'$ be two cocontinuous strict morphisms of pointed right derivators. Then a $1$-simplex
$\Psi$  in $\eqDER$ with $d_{1}\Psi=\phi$ and $d_{0}\Psi=\phi'$ induces a simplicial homotopy 
$\s_{\bullet} \phi \simeq \s_{\bullet} \phi'\colon \s_{\bullet} \der \to \s_{\bullet} \der'$.
\end{lem}
\begin{proof}
The idea is analogous to the definition of the simplicial enhancement in Proposition \ref{$K$-theory-as-a-simplicial-fun}. 
The required homotopy $\s_{\bullet} \phi \simeq \s_{\bullet}\phi'$ is a map $\s_{\bullet} \der\times\Delta[1] \to \s_{\bullet} \der'$ 
which we will specify by defining a natural transformation as follows
$$\xymatrix@C=30pt@R=8pt{&\Delta^{\op}\ar@/^8pt/[rd]^{\s_{\bullet} \der}\ar@{=>}[dd]_{\alpha}&
\\(\Delta\downarrow[1])^{\op}\ar@/^8pt/[ru]^{\text{source}}\ar@/_8pt/[rd]_{\text{source}}&&\operatorname{\mathtt{Set}},\\
&\Delta^{\op}\ar@/_8pt/[ru]_{\s_{ \bullet} \der'}&}$$
Recall that $\Psi$ is a strict cocontinuous morphism $\Psi\colon\der\r\der'([1]\times-)_{\eq}$. Given an object 
$\sigma \colon[k]\r[1]$ in $\Delta\downarrow[1]$, we define $\alpha(\sigma)$ to be the (co)restriction of the map on objects 
that comes from the following functor:
$$\xymatrix{
\der(\Ar[k])\ar[d]^{\Psi(\Ar[k])}\\
\der'([1]\times\Ar[k])_{\eq} \ar[d]^{(p\times\Ar[k])^{*}}\\
\der'(\Ar[1]\times\Ar[k])_{\eq} \ar[d]^{(\Ar(\sigma),\id{\Ar[k]})^{*}}\\
\der'(\Ar[k])
}$$
Here $p\colon \Ar[1] \to [1]$ is the functor defined by
$p(0,0) = 0$, 
$p(0,1) = 1$, 
$p(1,1) = 1$. 
The restriction of the composite functor to the $\s_{\bullet}$-construction is well-defined 
because the $1$-simplex $\Psi$ is in $\eqDER$ (and not merely in $\DER$). The naturality of $\alpha$ is 
straightforward to check.
\end{proof}

\begin{rem}\label{rem s-constr}
Moreover the functor $\der \mapsto \s_{\bullet} \der$ extends to a simplicial functor from $\eqDER$ to $\SSET$. 
The same argument works with $[1]$ replaced more generally by $[n]$ and $p$ by the functor $\Ar[n] \to [n]$, 
$(i \to j) \mapsto j$.
\end{rem}

An immediate consequence is the following invariance under strong equivalences.

\begin{cor} \label{invariance s-constr}
Let $\phi\colon \der \to \der'$ be a strong equivalence of pointed right derivators.
Then the induced maps $|\s_{\bullet}\phi|\colon |\s_{\bullet} \der| \to |\s_{\bullet} \der'|$ 
and $\wk(\phi)\colon \wk(\der) \to \wk(\der')$ are homotopy equivalences.
\end{cor}

We can now return to the proof of Proposition \ref{s-vs-S}.

\begin{proof}[Proof of Proposition \ref{s-vs-S}] Since
$$|([n], [m]) \mapsto \S_{n, m} \der| \qquad \cong \qquad |[m] \quad\mapsto \quad |[n] \mapsto \s_n \der([m]\times-)_{\eq}||,$$
it suffices to show that every simplicial operator in the $m$-direction
is a weak equivalence of simplicial sets after realizing in the $n$-direction.
This follows from Corollary \ref{invariance s-constr} and Example \ref{path-objects}.
\end{proof}

\subsection{Agreement with Waldhausen $K$-theory} The agreement of $\WK$ with the Waldhausen $K$-theory of well-behaved 
Waldhausen categories is based on results about the homotopically flexible variations of the $\S_{\bullet}$-construction 
by Blumberg--Mandell \cite{lsaKttKt} and Cisinski \cite{ikted}. We recall that Waldhausen's original $\S_{\bullet}$-construction 
of a Waldhausen category $\mathcal{C}$ \cite{akts} is a simplicial Waldhausen category $[n] \mapsto \S_n \mathcal{C}$, where the objects 
of $\S_n \mathcal{C}$ are given by diagrams $F \colon \Ar[n] \to \mathcal{C}$ such that $F(i \to i)$ is the zero object for all 
$i \in [n]$, and for every $i \leq j \leq k$, the square
\[
 \xymatrix{
 F(i \to j) \ar[r] \ar[d] & F(i \to k) \ar[d] \\
 F(j \to j) \ar[r] & F(j \to k)
 }
\]
has cofibrations as horizontal maps and is required to be a pushout. Restricting degreewise 
to the subcategory of (pointwise) weak equivalences gives a simplicial category $[n] \mapsto w \S_n \mathcal{C}$. 
We denote by $N_{\bullet} w \S_n \mathcal{C}$ the nerve of $w \S_n \mathcal{C}$. Then the Waldhausen $K$-theory 
of $\mathcal{C}$ is defined to be the space $K(\mathcal{C})\colon = \Omega | N_{\bullet} w S_{\bullet} \mathcal{C} |$.

\begin{thm} \label{agreement}
Let $\mathcal{C}$ be a strongly saturated derivable Waldhausen category. Then there is a natural weak equivalence
$$K(\mathcal{C}) \stackrel{\sim}{\To} \WK(\der(\mathcal{C})).$$
\end{thm}
\begin{proof}
The map is induced by a bisimplicial map
$$N_m w \S_n \mathcal{C} \To \S_{n,m} \der(\mathcal{C})$$
which sends an element $[m] \times \Ar[n] \to \mathcal{C}$ in $N_m w \S_n \mathcal{C}$ to the corresponding
object of $\S_{n,m} \der$. Since $\mathcal{C}$ is strongly saturated, so are also the Waldhausen categories
$\S_n \mathcal{C}$ for every $n$. It follows that the bisimplicial set $\S_{\ast, \bullet} \der(\mathcal{C})$ is
isomorphic to the bisimplicial set $N_{\bullet} w S^h_{\ast} \mathcal{C}$ of \cite{ikted} (and also to the 
bisimplicial set $N_{\bullet} w S'_{\ast} \mathcal{C}$ of \cite{lsaKttKt} since every map can be replaced by a cofibration). 
Then the result follows from the agreement of the $\S^h_{\bullet}$-construction with the $\S_{\bullet}$-construction, 
see \cite[Proposition 4.3]{ikted} (cf.~\cite[Theorem 2.9]{lsaKttKt} under the assumption that factorizations are 
functorial).
\end{proof}

\section{Derivator $K$-theory}\numberwithin{equation}{subsection}

\subsection{Recollections and $2$-categorical properties}
Derivator $K$-theory was first defined for triangulated derivators by Maltsiniotis in \cite{ktdt}. The definition, however,
applies similarly to all pointed right derivators. Here we will consider the explicit model defined in terms of the 
$\S_{\bullet}$-construction which was introduced by Garkusha \cite{sdckt1, sdckt2}, who also showed that it is equivalent to Maltsiniotis's in the triangulated setting.

\begin{defn}
The \emph{derivator $K$-theory} of a pointed right derivator $\der$ is defined to be the space
$\DK(\der) := \Omega |N_{\bullet} \iso \S_{\bullet} \der|.$
\end{defn}

Since a cocontinuous strict morphism $\phi\colon \der\r\der'$ perserves cocartesian squares and zero objects, it
is can be easily  checked that derivator $K$-theory defines a functor from $\eqs1Der$ to the category of topological 
spaces. Moreover, it is invariant under equivalences of derivators. 

\begin{prop}
If the strict morphism $\phi \colon \der \to \der'$ is an equivalence of pointed right derivators, then the 
induced map $\DK(\phi) \colon \DK(\der) \to \DK(\der')$ is a weak equivalence.
\end{prop}
\begin{proof} This is an immediate consequence of the fact that the geometric 
realization of simplicial categories sends pointwise (weak) equivalences to 
weak equivalences of spaces.
\end{proof}

We emphasize that an equivalence of right pointed derivators does not necessarily admit a \emph{strict} inverse. This means that an
equivalence in $\Der$ is not in general an ($2$-categorical) equivalence in $\sDer$. However, the two concepts are closely related as morphisms of prederivators can be strictified up
to a strict equivalence in $\PreDer$, see \cite[Proposition 10.14]{adkt}. 

Derivator $K$-theory is compatible with the $2$-categorical structure of $\eqsDer$. We show next how to enhance derivator 
$K$-theory to a simplicial functor from $N_{\bullet}\eqsDer$ to $\TOP$.  

\begin{prop}\label{2functor}
Derivator $K$-theory extends to a simplicial functor
$$\DK \colon N_{\bullet}\eqsDer \To \TOP.$$
\end{prop}
\begin{proof}
It suffices to construct a simplicial enhancement for the ($1$-)functor 
\begin{align*}
\eqs1Der &\To \mathtt{SSet},\\
\der&\;\mapsto \;\diag N_{\bullet}\iso\S_{\bullet}\der.
\end{align*}
Suppose we are given an $n$-simplex in  $\eqsDer(\der, \der')$
$$\alpha = (\phi_0 \stackrel{\alpha_1}{\Longrightarrow} \phi_1 \stackrel{\alpha_2}{\Longrightarrow} \cdots \stackrel{\alpha_n}{\Longrightarrow} \phi_n)$$
where
$$\xymatrix@C=30pt{\der\ar@/^15pt/[r]^{\phi_{k-1}}_{}="a"\ar@/_15pt/[r]_{\phi_{k}}^{}="b"&\der'\ar@{=>}"a";"b"_{\alpha_k}}$$
are invertible modifications. 
We construct a simplicial map
$$(\diag N_{\bullet} \iso \S_{\bullet})(\alpha) \colon (\diag N_{\bullet} \iso \S_{\bullet} \der) \times \Delta[n] \to (\diag N_{\bullet} \iso \S_{\bullet} \der')$$
by defining a natural transformation as follows
$$\xymatrix@C=30pt@R=8pt{&\Delta^{\op}\ar@/^8pt/[rd]^{\diag N_{\bullet} \iso\S_{\bullet}\der}\ar@{=>}[dd]_{\alpha_{*}}&
\\(\Delta\downarrow[n])^{\op}\ar@/^8pt/[ru]^{\text{source}}\ar@/_8pt/[rd]_{\text{source}}&&\operatorname{Set}\\
&\Delta^{\op}\ar@/_8pt/[ru]_{\diag N_{\bullet} \iso\S_{\bullet}\der'}&}$$
Given $\sigma\colon[k]\r[n]$ in $\Delta$, the map
$$\alpha_{*}(\sigma)\colon N_{k} \iso \S_{k}\der\To N_{k} \iso \S_{k}\der'$$
is defined as follows. Let 
$$\beta=\sigma^{*}(\alpha) = (\psi_0 \stackrel{\beta_1}{\Rightarrow} \psi_1 \stackrel{\beta_2}{\Rightarrow} \cdots \stackrel{\beta_k}{\Rightarrow} \psi_k)$$
Consider an element in the domain of $\alpha_*(\sigma)$, denoted by $(f_{1},\dots, f_{k})$, which is a chain of $k$ composable 
isomorphisms in $\S_{k}\der\subset \der(\Ar[k])$,
$$\xymatrix@C=10pt{X_{0}\ar[r]& \cdots\ar[r] &  X_{r-1}\ar[rr]^-{f_{r}}&&  X_{r}\ar[r]& \cdots\ar[r]&  X_{k}.}$$
The $k$-simplex $\beta$ gives rise to a $k\times k$ grid of commutative squares of solid arrows in $\S_{k}\der'\subset \der'(\Ar[k])$
$$\xymatrix@R=20pt@C=10pt{\psi_{0}(X_{0})\ar[r]\ar[d]\ar@{-->}[rd]& \cdots\ar[r] &  \psi_{0}(X_{r-1})\ar[rr]^-{\psi_{0}(f_{r})}\ar[d]&&  \psi_{0}(X_{r})\ar[r]\ar[d]& \cdots\ar[r]&  \psi_{0}(X_{k})\ar[d]\\
\vdots\ar[d]&\ddots\ar@{-->}[rd]&\vdots\ar[d]&&\vdots\ar[d]&&\vdots\ar[d]\\
\psi_{r-1}(X_{0})\ar[r]\ar[dd]_{\beta_{r}(X_{0})}& \cdots\ar[r] &  \psi_{r-1}(X_{r-1})\ar[rr]^-{\psi_{r-1}(f_{r})}\ar[dd]_{\beta_{r}(X_{r-1})}\ar@{-->}[rrdd]&&  \psi_{r-1}(X_{r})\ar[r]\ar[dd]^{\beta_{r}(X_{r})}& \cdots\ar[r]&  \psi_{r-1}(X_{k})\ar[dd]^{\beta_{r}(X_{k})}\\
&&&&&&\\
\psi_{r}(X_{0})\ar[r]\ar[d]& \cdots\ar[r] &  \psi_{r}(X_{r-1})\ar[rr]_-{\psi_{r}(f_{r})}\ar[d]&&  \psi_{r}(X_{r})\ar[r]\ar[d]\ar@{-->}[rd]& \cdots\ar[r]&  \psi_{r}(X_{k})\ar[d]\\
\vdots\ar[d]&&\vdots\ar[d]&&\vdots\ar[d]&\ddots\ar@{-->}[rd]&\vdots\ar[d]\\
\psi_{k}(X_{0})\ar[r] & \cdots\ar[r] &  \psi_{k}(X_{r-1})\ar[rr]^-{\psi_{k}(f_{r})} &&  \psi_{k}(X_{r})\ar[r] & \cdots\ar[r]&  \psi_{k}(X_{k}) }$$
We set $\alpha_{*}(\sigma)(f_{1},\dots, f_{k})$ to be the sequence of $k$ diagonal morphisms, depicted as dashed arrows,
$$\xymatrix@C=15pt{\psi_{0}(X_{0})\ar[r]& \cdots\ar[r] &  \psi_{r-1}(X_{r-1})\ar[rrrr]^-{\beta_{r}(X_{r})\psi_{r-1}(f_{r})}_-{\psi_{r}(f_{r})\beta_{r}(X_{r-1})}&&&&  \psi_{r}(X_{r})\ar[r]& \cdots\ar[r]&  \psi_{k}(X_{k}).}$$
The naturality of $\alpha_{*}$ in $\sigma$ is straightforward. For the compatibility with composition, we consider an $n$-simplex in $N_{\bullet} \eqsDer (\der', \der'')$,
$$\alpha' = (\phi'_0 \stackrel{\alpha'_1}{\Longrightarrow} \phi'_1 \stackrel{\alpha'_2}{\Longrightarrow} \cdots \stackrel{\alpha'_n}{\Longrightarrow} \phi'_n)$$
and then it suffices to check that for all $\sigma: [k] \to [n]$ in $\Delta \downarrow [n]$, we have
$$(\alpha' \alpha)_*(\sigma) = \alpha'_*(\sigma) \alpha_*(\sigma).$$
Indeed if $\beta'=\sigma^{*}(\alpha') = (\psi'_0 \stackrel{\beta'_1}{\Rightarrow} \psi'_1 \stackrel{\beta'_2}{\Rightarrow} \cdots \stackrel{\beta'_k}{\Rightarrow} \psi'_k)$
then each of the maps above, when applied to an element $(f_1, \dots, f_k) \in N_k \iso \S_k \der$, gives 
$$\xymatrix{
\psi'_0\psi_{0}(X_{0})\ar[r] & \cdots\ar[r] &  \psi'_{r-1}\psi_{r-1}(X_{r-1})
\ar[d]_-{(\beta'_r\beta_{r})(X_{r})(\psi'_r\psi_{r-1})(f_{r})}^-{\beta'_r(\psi_r(X_r))\psi'_{r-1}(\psi_{r}(f_r)\beta_{r}(X_{r-1}))} && \\
& & \psi'_{r}\psi_r(X_{r})\ar[r]& \cdots\ar[r]&  \psi'_k\psi_{k}(X_{k}).
}
$$
This $k$-simplex can be obtained as a diagonal in a $3$-dimensional cube, in the same way that $\alpha_{*}(\sigma)(f_{1},\dots, f_{k})$ is a diagonal in a square. Therefore, the vertical map can be written in six different ways. We have just chosen two of them.
\end{proof}

\begin{rem}
The proposition shows that the homotopy class of the morphism $\DK(\phi) \colon \DK(\der) \to \DK(\der')$ depends only on the 
isomorphism class of $\phi \colon \der \to \der'$ in the $2$-category $\eqsDer$. This together with the invariance of derivator $K$-theory 
under equivalences implies that derivator $K$-theory is in fact functorial in the homotopy category of spaces with respect 
to \emph{all} morphisms of derivators. More precisely, if for a category $\mathcal G$ which is enriched in groupoids, we denote by  
$\pi_0\mathcal G$ the $1$-category obtained by identifying isomorphic morphisms, then there exists a unique factorization
$$\xymatrix{\pi_0\eqsDer\ar[r]^{\DK}\ar[d]_{\pi_0\rho}&\top/\!\simeq\\
\pi_0\eqDer\ar[ru]}$$
Here $\top/\!\simeq$ is the homotopy category of topological spaces. Compare \cite[Corollary 10.19]{adkt}.  
\end{rem}

\begin{rem} \label{coherent-eq-vs-eq}
Derivator $K$-theory $\DK(\der)$ is weakly equivalent to the geometric realization of Waldhausen $K$-theories 
$$|[n] \mapsto \wk(\iso_n \der)|.$$
Note that the derivators $\{\iso_n \der\}_{n \geq 0}$ are equivalent in $\Der$ and the simplicial operators are 
equivalences of derivators. Therefore, given that Waldhausen and derivator $K$-theory are different in general 
\cite{ankttd}, it follows that Waldhausen $K$-theory is not invariant under equivalences of derivators. In particular, 
it follows that there are equivalences of (pre)derivators which are not strong, and more specifically, there are 
derivators $\der$ such that the canonical `degeneracy' equivalence $\der \to \iso_n \der$ is not a strong equivalence. 
\end{rem}

In the case where $\der = \der(\mathcal{C})$ for some derivable Waldhausen category 
$(\mathcal{C}, co \mathcal{C}, w \mathcal{C})$, the following variant of derivator $K$-theory is available. 
Passing to the homotopy categories of the $\S_{\bullet}$-construction, we obtain a new simplicial category, 
$[n] \mapsto \ho{\S_{n} \mathcal{C}}$,
and a canonical morphism of simplicial categories $\ho{\S_{\bullet} \mathcal{C}} \to \S_{\bullet} \der(\mathcal{C})$.
This is degreewise an equivalence of categories and therefore the induced map 
$$\Omega|N_{\bullet} \iso \ho{\S_{\bullet} \mathcal{C}}| \stackrel{\sim}{\To} \DK(\der(\mathcal{C}))$$
is a weak equivalence. 

\subsection{Comparison with Waldhausen $K$-theory} There is a natural comparison map from Waldhausen to derivator $K$-theory. 
For $n,m \geq 0$, the functors 
$$\dia_{[m], \Ar[n]} \colon \der([m] \times \Ar[n]) \to \cat([m], \der(\Ar[n]))$$ 
assemble to define a bisimplicial map
$$\S_{\bullet \bullet} \der \To N_{\bullet} \iso \S_{\bullet} \der$$
which then induces the comparison map from Waldhausen $K$-theory to derivator $K$-theory (cf.~\cite{ktdt,sdckt2,ankttd})
$$\mu \colon \WK(\der) \To \DK(\der).$$
We note that composing with the weak equivalence $\iota\colon \wk(\der) \to \WK(\der)$, we obtain
\[
\xymatrix{
\wk(\der) \ar[rd]_{\muob} \ar[r]^-{\iota}_-{\sim} & \WK(\der) \ar[d]^{\mu} \\
& \DK(\der)
}
\]
where the map $\muob$ is given degreewise simply by the inclusion of objects. The comparison maps $\mu$ and $\muob$ define natural 
transformations. Moreover, $\mu$ defines a natural transformation of simplicially enriched functors
$$
\xymatrix@R=15pt{
\eqDER \ar[rd]^{\WK}_-{\;}="a" \ar[dd]_{\rho}^-{\;}="b" & \\
& \TOP \\
N_{\bullet}\eqsDer \ar[ru]_{\DK}^-{\;}="b" &  \ar@{=>}"a";"b"_{\mu}
}
$$
and the same holds for $\muob$, cf.~Lemma \ref{s-construction} and Remark \ref{rem s-constr}. However, making use of these simplicial 
enrichments will not be required in what follows since it is possible to think of them, in a homotopical fashion, only as asserting 
certain invariance properties. We will concentrate instead on the natural transformation
$$
\xymatrix@C=40pt{
\eqs1Der\ar@/^15pt/[r]^{\wk}_-{\;}="a" \ar@/_15pt/[r]_{\DK}^-{\;}="b" &   \top \\ \ar@{=>}"a";"b"_{\muob}
}
$$
because this is technically a more convenient model of the comparison map for the statement of our results. Here $\top$ is the ordinary 
category of topological spaces.

In connection with the diagram above, it is interesting to mention that To\"{e}n--Vezzosi \cite{ktsc} gave a neat abstract argument, based only 
on functoriality, to show that Waldhausen $K$-theory cannot factor through $N_\bullet \eqsDer$ by a functor which is invariant under equivalences 
of derivators.

Maltsiniotis \cite{ktdt} conjectured that $\mu$ is a weak equivalence when $\der$ is the triangulated derivator associated with 
an exact category \cite{dtce}. This conjecture remains open, but several relevant results are known. Garkusha \cite{sdckt1}, 
based on previous results by Neeman on the $K$-theory of triangulated categories, showed that $\mu$ admits a retraction when 
$\der$ arises from an abelian category. Maltsiniotis \cite{ktdt} and Muro \cite{malt} showed
that $\mu$ induces an isomorphism on $\pi_0$ and $\pi_1$, respectively, for any $\der$ that arises from a strongly saturated 
derivable Waldhausen category. In \cite{ankttd}, we showed that $\mu$ fails to be a weak equivalence in general for triangulated derivators
that arise from differential graded algebras (or stable module categories). Moreover, we showed that the conjecture fails if 
derivator $K$-theory satisfies \emph{localization}, a property also conjectured by Maltsiniotis \cite{ktdt}. 

However, the pair $(\DK, \mu)$ turns out to be the best approximation to Waldhausen $K$-theory by a functor which
sends equivalences of derivators to weak equivalences. We choose a rather \emph{ad hoc} but direct way of 
formulating this property precisely as follows. 

First, in order to ensure that our categories remain locally small and so to avoid set-theoretical troubles, we fix a 
(small) set $S$ of pointed right derivators $\der$ closed under taking $\iso \der$, and restrict to the full subcategory of $\eqs1Der$ 
spanned by $S$, that we still denote by $\eqs1Der$. Second, it will be more convenient to work here with simplicial techniques 
and the delooped versions of Waldhausen and derivator $K$-theory. Thus we set:
$$\Omega^{-1} \WK (\der) \colon = \diag \S_{\bullet \bullet} \der,$$
$$\Omega^{-1} \wk (\der) \colon = \s_{\bullet} \der,$$
$$\Omega^{-1} \DK(\der) \colon = \diag \iso_{\bullet} \S_{\bullet} \der,$$
and we have a natural transformations $\mu \colon \Omega^{-1} \WK \Rightarrow \Omega^{-1} \DK$ and $\muob \colon \Omega^{-1} \wk \Rightarrow \Omega^{-1} \DK$ and a natural weak equivalence
$\iota \colon \Omega^{-1} \wk \Rightarrow \Omega^{-1}\WK$ with $\muob=\mu\iota$. 

\begin{defn} 
Let $\SSet^{\eqs1Der}$ be the functor category. The \emph{category $\facto$ of invariant approximations} to Waldhausen $K$-theory $\Omega^{-1}\wk$ is the full subcategory of the comma category $\Omega^{-1}\wk\downarrow\SSet^{\eqs1Der}$ spanned by the objects $\eta\colon\Omega^{-1}\wk\rr F$ such that $F\colon  \eqs1Der \to \SSet$ sends equivalences of derivators to weak equivalences. A morphism 
$$\xymatrix{&\Omega^{-1}\wk\ar@{=>}[ld]_\eta\ar@{=>}[rd]^{\eta'}&\\
F\ar@{=>}[rr]_u&&F'}$$
in $\facto$ is a \emph{weak equivalence} if the components of $u$ are weak equivalences of simplicial sets.
\end{defn} 

Note that $\muob\colon \Omega^{-1}\wk\rr \Omega^{-1}\DK$ is an object of $\facto$.  Following \cite{hlfmchc}, we say that an object $X$ of a category with weak equivalences 
$(\mathcal{C}, \mathcal{W})$ (satisfying in addition the ``2-out-of-6'' property) 
is \emph{homotopically initial} if there are homotopical functors $F_0, F_1 \colon \mathcal{C} \to \mathcal{C}$ and a natural transformation 
$f \colon F_0 \Rightarrow F_1$ such that: (i) $F_0$ is naturally weakly equivalent to the constant functor at $X$, (ii) $F_1$ is 
naturally weakly equivalent to the identity functor on $\mathcal{C}$ and (iii) $f_X \colon F_0(X) \to F_1(X)$ is a weak equivalence. If $X$ is initial in $\mathcal{C}$,
then it is also homotopically initial in this sense. If $X$ is homotopically initial in $\mathcal{C}$, then $X$ is initial in $\ho{\mathcal{C}}$. Finally,
the category of homotopically initial objects in $(\mathcal{C}, \mathcal{W})$ is either empty or homotopically contractible. We 
refer the reader to \cite{hlfmchc} for more details. 

\begin{thm}\label{initial}
The object $\muob\colon \Omega^{-1}\wk\rr \Omega^{-1}\DK$ is homotopically initial in the category with weak equivalences $\facto$. 
\end{thm}
\begin{proof}
 Let $F \colon  \eqs1Der \to \SSet$ be 
a functor. Then there is a canonical way of associating to $F$ a new functor
\begin{align*}
\mathbb{H}F \colon  \eqs1Der &\To \SSet,\\ 
\der &\;\mapsto\; \diag F(\iso_{\bullet} \der).
\end{align*}
The inclusion of 
$0$-simplices defines a natural transformation $$\iota_F\colon F \Longrightarrow \mathbb{H}F.$$  
By definition, we have $$\iota_{\Omega^{-1}\wk}=\muob\colon \Omega^{-1}\wk\Longrightarrow \mathbb H \Omega^{-1}\wk=\Omega^{-1}\DK.$$
If $F$ sends equivalences of derivators to weak 
equivalences then the simplicial operators of $F(\iso_{\bullet} \der)$ in the $\iso_\bullet$-direction are weak equivalences, 
so $\iota_F$ is a natural weak equivalence. In this case, it follows that $\mathbb{H}F$ also sends equivalences of derivators to weak 
equivalences. Using this fact, we can view the $\mathbb H$-construction as an endofunctor, denoted $\widetilde{\mathbb H} \colon\facto\r\facto$, 
which sends an object $\eta\colon \Omega^{-1}\wk\rr F$ in $\facto$, to the natural transformation $\widetilde{\mathbb H}(\eta)$ given by the diagonal in 
the following commutative square
$$\xymatrix{\Omega^{-1}\wk\ar@{=>}[r]^-\eta\ar@{=>}[d]_\muob&F\ar@{=>}[d]^{\iota_{F}}_-\sim\\
\Omega^{-1}\DK\ar@{=>}[r]^-{\mathbb H\eta}&\mathbb H F}$$
The natural transformation $\iota$ induces a natural weak equivalence $\iota'\colon\id{\facto}\rr\widetilde{\mathbb H}$ given by the right 
vertical arrow, and the bottom horizontal arrow defines a natural transformation from the constant functor at
$\muob\colon \Omega^{-1}\wk\rr \Omega^{-1}\DK$ to $\widetilde{\mathbb H}$. Hence, the result follows.
\end{proof}

We wish to remark that we could have worked entirely with simplicially enriched categories and functors in this section. More specifically, the construction 
$\mathbb H F$ in the proof of the last theorem has a simplicial enhancement which can be constructed as in the proof of Proposition \ref{2functor}. We decided 
to work with $1$-categories in order to avoid the ensuing technicalities.

\section{Some open questions}\numberwithin{equation}{subsection}

\subsection{Derivators and the homotopy theory of homotopy theories}\label{renaudin} The simplicial enrichment of the category of derivators 
leads to a homotopy theory of derivators which is more discerning than the $2$-categorical one and is closer to the homotopy
theory of categories with weak equivalences. An interesting problem is to understand exactly how close this relationship is,
and find out whether this homotopy theory of derivators is rich and structured enough to be (or contain a part of) a model for 
the homotopy theory of homotopy theories. In the case of the $2$-category of derivators, a theorem of Renaudin \cite{pcthqd} 
specified the relationship between combinatorial model categories and their associated derivators (see also Appendix \ref{appA}
for a review). In this context, the question would be whether this result can be improved in view of the simplicial enrichment 
of derivators. The results of Appendix \ref{appA} may be a first step towards this direction.

\subsection{Derived equivalences vs.~strong equivalences} We do not know whether an exact functor of well-behaved Waldhausen categories
which is a derived equivalence also induces a strong equivalence between the associated pointed right derivators. This is clear in the 
case where the derived equivalence admits a homotopy inverse (cf.~Remark \ref{homotopy-eq-coherent-eq}), but such an inverse may not exist 
strictly at the level of models in general. If the statement is true, then we will be able to deduce the invariance of Waldhausen $K$-theory 
(of Waldhausen categories) under derived equivalences also from the invariance of Waldhausen $K$-theory of pointed right derivators under 
strong equivalences. 


\subsection{Additivity for derivator $K$-theory}

The additivity of derivator $K$-theory was proved by Cisinski and Neeman for triangulated derivators \cite{adkt}. However, the 
more general case of additivity for derivator $K$-theory of pointed right derivators seems to remain an open problem. We emphasize 
that this seems to be so also in the case where the derivator admits a model. 
In this case, we tried to apply Waldhausen's original proof and generalize the approach in \cite{sdckt2}, but we discovered a 
gap in the proof of \cite[Theorem 3.1]{sdckt2} which we could not fix. (Namely, in diagram (7), at the bottom of page 655, the 
arrow $\varphi^*_{X_i}v_ic_i\colon V_i'' \r\bar{X}_i$ need not be a weak equivalence.) In particular, we do not know whether derivator 
$K$-theory of pointed right derivators is invariant under an appropriately defined notion of stabilization which would produce a triangulated 
derivator. 

A related problem is to show that additivity holds for the Waldhausen $K$-theory of pointed right derivators. Of course, this
is true if the derivator admits a model. However it would still be interesting to establish the general case as it is through 
this generality that the concept of derivator can also be tested.

\appendix
\numberwithin{equation}{section}

\section{Combinatorial model categories and derivators}\label{appA}

The purpose of this appendix is to highlight some results on the connections between combinatorial model categories and 
derivators. Since the discussion is heavily based on Renaudin's paper \cite{pcthqd}, we will give a very concise 
review of his results while providing precise references where necessary. Then we will record some minor strengthenings 
of Renaudin's main theorem with a view to addressing the questions of \ref{renaudin}. 

Let $\MOD$ denote the $2$-category of left proper combinatorial model categories, Quillen adjunctions and natural transformations
between left Quillen functors. Following \cite{pcthqd}, we view the morphism categories as categories with weak equivalences
where the weak equivalences are given by \textit{Quillen homotopies}. We recall that a natural transformation of left Quillen functors
is a Quillen homotopy if it is pointwise a weak equivalence at the cofibrant objects (see \cite[D\'{e}finition 2.1.2]{pcthqd}). 
Passing to the homotopy categories of all morphism categories yields a new $2$-category $\underline{\MOD}$. We note that 
$\underline{\MOD}$ is enriched in the category of all categories $\CAT$. 

The class of Quillen equivalences in $\underline{\MOD}$ admits a calculus of right fractions \cite[Proposition 2.3.2]{pcthqd}. Thus
the bilocalization of $\underline{\MOD}$ at the class of Quillen equivalences exists, denoted here by $\underline{\MOD}[\mathcal{Q}^{-1}]$, and 
is actually equivalent to the bilocalization of $\MOD$ at the class of Quillen equivalences \cite[Th\'eor\`eme 2.3.3]{pcthqd}. 

Let $\MOD^p$ be the $1$- and $2$-full subcategory of presentable model categories, that is, combinatorial model categories that arise from a 
left Bousfield localization of the projective model category of $C$-diagrams in $\SSet$, for some small category $C$, at a set of 
morphisms $S$. Every combinatorial model category is equivalent to a presentable one \cite{cmchp}. Presentable model categories have certain nice 
`cofibrancy' properties which can in particular be used to show that the $2$-functor
$$\underline{\MOD}^p \To \underline{\MOD}[\mathcal{Q}^{-1}]$$
is a biequivalence \cite[Proposition 2.3.4]{pcthqd}. Here $\underline{\MOD}^p$ denotes the corresponding $1$- and $2$-full subcategory of 
$\underline{\MOD}$. The restriction to presentable model categories in what follows is mainly a technical matter and owes essentially to the 
rigidity of $\MOD$ compared say to the essentially equivalent context of presentable $\infty$-categories.

Let $\BigDer$ (resp.~ $\BigDer_{!}$, $\BigDer_{ad}$) denote the $2$-category of derivators with domain $\Dia=\cat$ and values 
in the $2$-category $\CAT$ together with pseudo-natural transformations (resp.~ cocontinuous morphisms, adjunctions between 
derivators) as $1$-morphisms, and modifications as $2$-morphisms. Cisinski \cite{idccm} constructed a pseudo-functor 
$$\mathbb{D}(-) \colon \MOD \to \BigDer_{ad}$$
which is defined on objects by $\mathcal{M} \mapsto \der(\mathcal{M})$ (cf.~Section \ref{examples}) and sends Quillen equivalences to equivalences of 
derivators. We note that $\der(\mathcal{M})$ takes values in locally small categories. There is an induced pseudo-functor of $2$-categories
$$\underline{\mathbb{D}}(-) \colon \underline{\MOD}[\mathcal{Q}^{-1}] \to \BigDer_{ad}.$$
Renaudin \cite{pcthqd} showed that this functor is a \emph{local equivalence}, i.e.~it induces equivalences between the 
morphism categories \cite[Th\'{e}or\`{e}me 3.3.2]{pcthqd}. This could be interpreted as identifying a part of $\BigDer_{ad}$ 
with a truncation of the homotopy theory of homotopy theories as modelled by $\MOD$. For our purposes, it will be necessary 
to reformulate this result in terms of the larger $2$-category $\BigDer_{!}$ (cf.~\cite[Remarque 3.3.3]{pcthqd}).

\begin{thm} \label{coherent-renaudin-2}
The canonical pseudo-functor 
$$\underline{\der}(-) \colon \underline{\MOD}^p \To \BigDer_!$$
is a local equivalence. 
\end{thm}
\begin{proof}
Since the composition $\underline{\MOD}^p \to \underline{\MOD}[\mathcal{Q}^{-1}] \to \BigDer_{ad}$ is a local 
equivalence, it suffices to show that for all $\mathcal{M}$ and $\mathcal{N}$ in $\MOD^p$, the fully-faithful inclusion functor 
\begin{equation}\label{adjoint_functor_theorem}
	\BigDer_{ad}(\der(\mathcal{M}), \der(\mathcal{N})) \hookrightarrow \BigDer_!(\der(\mathcal{M}), \der(\mathcal{N}))
\end{equation}
is also essentially surjective. Let $F \colon \der(\mathcal{M}) \to \der(\mathcal{N})$ be a cocontinuous morphism. 
Suppose that $\mathcal{M} = L_S \SSet^C$. The Quillen adjunction $\mathrm{Id}\colon \SSet^C \rightleftarrows \mathcal{M} \colon \mathrm{Id}$
induces a morphism in $\BigDer_{ad}$, denoted as follows
$$\mathbb{L}_S(\mathrm{Id}) \colon \der(\SSet^C) \rightleftarrows \der(\mathcal{M})\colon \mathbb{R}_S(\mathrm{Id}).$$
The composite $F' = F \circ \mathbb{L}_S(\mathrm{Id})\colon \der(\SSet^C) \to \der(\mathcal{N})$ is a cocontinuous morphism.
By \cite[Remarque 3.3.3]{pcthqd} (or more directly, by using the universal property of $\SSet^C$ due to Dugger, 
see \cite[Proposition 2.2.7]{pcthqd}, and that of $\der(\SSet^C)$ due to Cisinski, see \cite[Th\'eor\`eme 3.3.1]{pcthqd}), 
there is a Quillen adjunction
$$\widetilde{F}' \colon \SSet^C \rightleftarrows \mathcal{N} \colon \widetilde{G}'$$
such that $\der(\widetilde{F}')$ is isomorphic to $F'$ in $\BigDer_!(\der(\SSet^C), \der(\mathcal{N}))$. Then the universal
property of Bousfield localization shows that $(\widetilde{F}', \widetilde{G}')$ descends to a Quillen adjunction
$$\widetilde{F}'' \colon L_S \SSet^C \rightleftarrows \mathcal{N} \colon \widetilde{G}''$$
such that $\der(\widetilde{F}'') \circ \mathbb{L}_S(\mathrm{Id}) $ is isomorphic to $F \circ \mathbb{L}_S (\mathrm{Id})$. Then 
$$\der(\widetilde{F}'') \colon \der(\mathcal{M}) \rightleftarrows \der(\mathcal{N}) \colon \der(\widetilde{G}'')$$ 
is an adjunction of derivators and the left adjoint $\der(\widetilde{F}'')$ is isomorphic to $F$ since the functor 
$$\mathbb{L}_S(\mathrm{Id})^* \colon \BigDer_!(\der(\mathcal{M}), \der(\mathcal{N})) \to \BigDer_!(\der(\SSet^C), \der(\mathcal{N}))$$
is fully faithful, see \cite[4.2-4.4]{hktui}. 
\end{proof}

We would like to emphasize that the equivalence of categories \eqref{adjoint_functor_theorem} in the last proof can be regarded as an 
adjoint functor theorem for derivators that arise from combinatorial model categories.

We recall from Groth \cite{monoidalder} the construction of internal hom-objects in the $2$-category of derivators. 
Given prederivators $\der, \der'\colon \cat \to \CAT$
there is a prederivator $\HOM(\der, \der'): \cat^{\op} \to \CAT$ which is defined explicitly by 
$$\HOM(\der, \der')(X) = \BigDer(\der, \der'_X).$$
Moreover, if $\der'$ is a derivator, then so is $\HOM(\der, \der')$, see \cite[Proposition 1.20]{monoidalder}. The simplicial enrichments of the previous sections are 
obtained from this by setting $X = [n]$ and restricting to the objects. If $\der$ and $\der'$ are derivators we also consider the following closely related prederivator
\begin{align*}
\HOM_!(\der, \der')\colon \cat^{\op} & \To \CAT,\\
X&\; \mapsto \;\BigDer_!(\der, \der'_X).
\end{align*}
To see that this is again a prederivator, it suffices to consider $u\colon X \to Y$ in $\cat$ and a cocontinuous morphism 
$\phi \colon \der \rightarrow \der'_{Y}$, and then note that the morphism 
$$\HOM(\der, \der')(u) (\phi) \colon = u^* \phi \colon \der \rightarrow \der'_{Y} \rightarrow \der'_{X}$$
is again cocontinuous because $u^*\colon \der'_Y \to \der'_X$ is cocontinuous (in fact, it admits a right adjoint 
$u_*\colon \der'_X \to \der'_Y$). 

Similarly, it is easy to check that $\HOM_!(\der, \der')$ is in fact a right derivator. For every 
$u\colon X \to Y$, the pullback functor defined above
$$u^*\colon \BigDer_!(\der, \der'_Y) \To \BigDer_!(\der, \der'_X)$$
admits a left adjoint
$$u_! \colon \BigDer_!(\der, \der'_X) \To \BigDer_!(\der, \der'_Y)$$
which is defined as for the derivator $\HOM(\der, \der')$: given a cocontinuous morphism $\phi \colon \der \rightarrow \der'_X$, then
$$u_!(\phi) \colon = u_!\phi \colon \der \rightarrow \der'_X \rightarrow \der'_Y.$$
We refer the reader to \cite[Propositions 2.5 and 2.9, Example 2.10]{derivators(groth)} for more details about 
adjunctions.

The purpose of this appendix is to show that the functor $\mathbb{D}(-)$ also preserves hom-objects in the sense of the
following theorem. For $\mathcal{M}$ and $\mathcal{N}$ in $\MOD$, let $\MOD_l(\mathcal{M}, \mathcal{N})$ denote the category of left Quillen functors $\mathcal{M} \to \mathcal{N}$
and natural transformations. This is again a category with weak equivalences, the Quillen homotopies, and the forgetful functor 
$\MOD(\mathcal{M}, \mathcal{N}) \to \MOD_l(\mathcal{M}, \mathcal{N})$ is an equivalence. 

\begin{thm} \label{coherent-renaudin}
Let $\mathcal{M}$ and $\mathcal{N}$ be presentable model categories. Then there is an equivalence of prederivators
$$\Phi(\mathcal{M}, \mathcal{N})\colon \der(\MOD_l(\mathcal{M}, \mathcal{N})) \simeq \HOM_!(\der(\mathcal{M}), \der(\mathcal{N})).$$
\end{thm}
\begin{proof}
For every small category $X$, there is a natural equivalence of categories
\[
 \der(\MOD_l(\mathcal{M}, \mathcal{N}))(X) \simeq \ho(\MOD_l(\mathcal{M}, (\mathcal{N}^X)_{\mathrm{inj}}))
\]
since a $X$-diagram of left Quillen functors $\mathcal{M} \to \mathcal{N}$ is the same as a left Quillen functor
$\mathcal{M} \to (\mathcal{N}^X)_{\mathrm{inj}}$ where the target is given the injective model structure. The latter 
model category is strictly speaking no longer presentable, but we can find a natural replacement for it by a presentable 
one $\mathcal{N}^X$ simply by a change to the projective (co)fibrations. Then by \cite[Proposition 2.2.9]{pcthqd}, we 
have an equivalence of categories 
$$\ho(\MOD_l(\mathcal{M}, (\mathcal{N}^X)_{\mathrm{inj}})) \stackrel{\simeq}{\longleftarrow} \ho(\MOD_l(\mathcal{M}, \mathcal{N}^X)).$$
There is a morphism of prederivators, induced by $\der(-)$, 
$$\Phi(\mathcal{M}, \mathcal{N})\colon \ho(\MOD_l(\mathcal{M}, (\mathcal{N}^{?})_{\mathrm{inj}}) \To \HOM_!(\der(\mathcal{M}), \der(\mathcal{N}))$$
whose components are equivalences of categories because we have commutative diagrams as follows
\[
\xymatrix@C=0pt{
\ho(\MOD_l(\mathcal{M}, (\mathcal{N}^X)_{\mathrm{inj}})) \ar[rr]^{\Phi(\mathcal{M}, \mathcal{N})_X} && \BigDer_!(\der(\mathcal{M}), \der(\mathcal{N}^X)) \\
& \ho(\MOD_l(\mathcal{M}, \mathcal{N}^X)) \ar[ur]_{\simeq} \ar[ul]^{\simeq} & \\
}
\]
where the indicated equivalence on the right is a consequence of Theorem \ref{coherent-renaudin-2}.
\end{proof}




\section{A remark on the approximation theorem}

The original approximation theorem of Waldhausen \cite{akts} states sufficient conditions for an exact functor of Waldhausen
categories to induce an equivalence in $K$-theory. Although Waldhausen did not analyse the meaning of these conditions from the
viewpoint of homotopical algebra, various authors have later studied connections between abstract homotopy theory and
Waldhausen $K$-theory and have shown more general and refined versions of the approximation theorem (see 
\cite{tt,aKt-model-cat,ktde,ikted,aKtaht}). These results ultimately say that Waldhausen $K$-theory is an invariant of 
homotopy theories and allow definitions of the theory via $\infty$-categories or simplicial categories (see also \cite{ktsc}).

\begin{thm}[Cisinski \cite{ikted}, Blumberg--Mandell \cite{aKtaht}]
Let $F\colon \mathcal{C} \to \mathcal{C}'$ be an exact functor of strongly saturated derivable Waldhausen categories. If 
the induced functor $\ho{F} \colon \ho{\mathcal{C}} \to \ho{\mathcal{C}'}$ is an equivalence of categories, then the map 
$ w\S_n F \colon  w\S_n \mathcal{C} \to  w\S_n \mathcal{C}'$ is a weak equivalence for all $n \geq 0$. In particular,
the map $K(F)\colon K(\mathcal{C}) \to K(\mathcal{C}')$ is also a weak equivalence.
\end{thm}

The purpose of this appendix is to note the following result which may be regarded as a partial converse to the approximation 
theorem. The proof is based on ideas of Dwyer-Kan for modelling mapping spaces in homotopical algebra via zigzag diagrams 
(see, e.g., \cite{fcha}) and related results from \cite{aKtaht}.

\begin{thm}
Let $F\colon \mathcal{C} \to \mathcal{C}'$ be an exact functor of derivable Waldhausen categories. Suppose that:
\begin{itemize} \item[(i)] $wF\colon w\mathcal{C} \to w \mathcal{C}'$ induces isomorphisms on $\pi_0$ and
$\pi_1$ for all basepoints,
\item[(ii)] $w \S_2 F \colon w \S_2 \mathcal{C} \to w \S_2 \mathcal{C}'$ is $1$-connected (i.e.~it induces an isomorphism
on $\pi_0$ and an epimorphism on $\pi_1$ for all basepoints).
\end{itemize}
Then $\ho{F} \colon \ho{\mathcal{C}} \to \ho{\mathcal{C}'}$ is an equivalence of categories.
\end{thm}
\begin{proof}
Consider the commutative square
$$\xymatrix{
w\S_2 \mathcal{C} \ar[rr]^{w \S_2 F} \ar[d]^{(\partial_1,\partial_2)} && w \S_2 \mathcal{C}' \ar[d]^{(\partial_1, \partial_2)} \\
w \mathcal{C} \times w \mathcal{C} \ar[rr]^{wF \times wF} && w \mathcal{C}' \times w \mathcal{C}'}$$ 
Using the properties of the long exact sequence of homotopy groups and  assumptions (i) and (ii), it follows that the induced 
map between the homotopy fibers of the vertical maps (at any basepoint) induces an isomorphism on $\pi_0$ (see, e.g., \cite[Lemma 1.4.7]{mcat}). 
Applying \cite[Theorem 1.2]{aKtaht}, the map between the homotopy fibers at the points defined by $(X_1, X_2) \in\ob \mathcal{C} \times \ob \mathcal{C}$ and $(F(X_1), F(X_2))\in \ob \mathcal{C}' \times \ob \mathcal{C}'$, 
respectively, can be identified with the map induced by $F$ between the corresponding mapping spaces in the respective hammock localizations
$$L^H(\mathcal{C})(X_1, X_2) \to L^H(\mathcal{C}')(F(X_1), F(X_2)).$$
Thus applying $\pi_0$ to this map gives an isomorphism 
$$\ho{F}\colon \ho{\mathcal{C}}(X_1, X_2) \cong \ho{\mathcal{C}'}(F(X_1), F(X_2))$$
and therefore $\ho{F}$ is fully faithful. It is also essentially surjective because $wF$ is an epimorphism on $\pi_0$.
\end{proof}

To sum up, we have the following

\begin{cor}
Let $F\colon \mathcal{C} \to \mathcal{C}'$ be an exact functor of strongly saturated derivable Waldhausen categories. If
\begin{itemize} \item[(i)] $wF\colon w\mathcal{C} \to w \mathcal{C}'$ induces isomorphisms on $\pi_0$ and $\pi_1$, and
\item[(ii)] $w \S_2 F \colon w \S_2 \mathcal{C} \to w \S_2 \mathcal{C}'$ is $1$-connected, 
\end{itemize}
then $w\S_n F \colon w\S_n \mathcal{C} \to w \S_n \mathcal{C}'$ is a weak equivalence for all $n \geq 0$. In particular,
the induced map $K(F)\colon K(\mathcal{C}) \to K(\mathcal{C}')$ is also a weak equivalence.
\end{cor}

These results show that being a derived equivalence is much stronger than being a $K$-equivalence. More specifically, the property
of being a derived equivalence does not take into account the ``group-completion'' process that takes place in the definition of 
$K$-theory. To obtain an ideal approximation theorem, that encodes this group completion process, one would need to `localize 
$\mathcal{C}$ and $\mathcal{C}'$' at all the relations which are derived from the additivity property, and then ask for the weaker 
property that these localized objects are equivalent. This localization is accomplished using $\infty$-categories with the construction 
of the universal additive invariant in \cite{BGT} and it is essentially shown that it is equivalent to Waldhausen $K$-theory. 

\providecommand{\bysame}{\leavevmode\hbox to3em{\hrulefill}\thinspace}
\providecommand{\MR}{\relax\ifhmode\unskip\space\fi MR }
\providecommand{\MRhref}[2]{%
  \href{http://www.ams.org/mathscinet-getitem?mr=#1}{#2}
}
\providecommand{\href}[2]{#2}


\end{document}